\documentclass[a4paper,reqno]{amsart}

\usepackage{amssymb}
\usepackage{latexsym}
\usepackage{amsmath}
\usepackage{amsthm}
\usepackage{euscript}

\usepackage{pgfplots}
\usepackage{mathrsfs}
\usetikzlibrary{arrows}
\usetikzlibrary{hobby}
\usetikzlibrary{calc}
\usepackage{amssymb}
\usepackage{tikz-3dplot}

\usepackage{mathtools}
\usepackage{changepage}
\usepackage{fge}
\usepackage{hyperref}
\usepackage{xcolor}
\usepackage[subnum]{cases}
\usepackage{tkz-euclide,amsmath}
\usetikzlibrary{quotes,angles}
\usetikzlibrary{intersections}
\pgfdeclarelayer{bg}
\pgfsetlayers{bg,main}
\usetikzlibrary{positioning}
\usetikzlibrary{calc}

\usepackage{pgfplots}
\pgfplotsset{compat=1.10}
\usepgfplotslibrary{fillbetween}
\usetikzlibrary{patterns}

\def\phi{\varphi}

\def\ee{{\mathrm e}}

\DeclareMathOperator\dom{{\text{\rm dom\,}}} 

\DeclareMathOperator{\Div}{div} 
\DeclareMathOperator{\diver}{div}

\DeclareMathOperator{\spann}{span}
\DeclarePairedDelimiter{\abs}{\lvert}{\rvert}
\DeclareMathOperator{\Real}{Re}

\newcommand\ga[1]{\Gamma_{#1}}
\DeclareMathOperator{\om}{{\Omega}}
\newcommand\bnd{{\partial \Omega}}

\newcommand\rd{{\mathbb{R}}^d}
\newcommand\lv[1]{\lambda_{#1}(V)} 
\newcommand\lvg[1]{\lambda^{\Gamma}_{#1}(V)} 
\newcommand\hu{H^1(\Omega)} 
\newcommand\huz{H^1_0(\Omega)} 
\newcommand\hmix{H^1_{0,\Gamma}(\Omega)} 
\newcommand\hd{H^2(\Omega)} 
\newcommand\qa{\mathfrak{a}} 
\newcommand\lspan{\mathrm{span}} 

\def\dist{{\text{\rm dist}}}

\DeclareMathOperator{\R}{\mathbb{R}}
\DeclareMathOperator{\C}{\mathbb{C}}

\renewcommand\ga{\Gamma}
\newcommand\gac{{\Gamma^c}}

\newcommand\ld[1]{L^2(#1)}

\renewcommand\setminus{\mathbin{\mathpalette\rsetminusaux\relax}}
\newcommand\rsetminusaux[2]{\mspace{-4mu}
  \raisebox{\rsmraise{#1}\depth}{\rotatebox[origin=c]{-20}{$#1\smallsetminus$}}
 \mspace{-4mu}
}
\newcommand\rsmraise[1]{%
  \ifx#1\displaystyle .8\else
    \ifx#1\textstyle .8\else
      \ifx#1\scriptstyle .6\else
        .45%
      \fi
    \fi
  \fi}

\usepackage{tikz}

\makeatletter
\newcommand*{\rom}[1]{\expandafter\@slowromancap\romannumeral #1@}
\makeatother

\definecolor{darkspringgreen}{rgb}{0.09, 0.45, 0.27}

\definecolor{darktangerine}{rgb}{1.0, 0.66, 0.07}
\definecolor{amber}{rgb}{1.0, 0.75, 0.0}

\definecolor{purple}{rgb}{0.6, 0.4, 0.8}

\definecolor{amber(sae/ece)}{rgb}{1.0, 0.49, 0.0}

\newtheorem{theorem}{Theorem}[section]
\newtheorem*{thm*}{Theorem}
\newtheorem{proposition}[theorem]{Proposition}
\newtheorem{corollary}[theorem]{Corollary}
\newtheorem{lemma}[theorem]{Lemma}

\theoremstyle{definition}

\theoremstyle{definition}
\newtheorem{definition}[theorem]{Definition}
\newtheorem{example}[theorem]{Example}

\newtheorem{remark}[theorem]{Remark}

\makeatletter
\def\th@newremark{\th@remark\thm@headfont{\bfseries}}
\makeatletter
\theoremstyle{newremark}

\numberwithin{equation}{section}

\title[]{Inequalities for eigenvalues of Schr\"odinger operators with mixed boundary conditions}

\author[N.~Aldeghi]{Nausica Aldeghi}
\address{Matematiska institutionen \\ Stockholms universitet \\
106 91 Stockholm \\
Sweden}
\email{nausica.aldeghi@math.su.se}

\begin{document}

\maketitle

\begin{abstract}
We consider the eigenvalue problem for the Schr\"odinger operator on bounded, convex domains with mixed boundary conditions, where a Dirichlet boundary condition is imposed on a part of the boundary and a Neumann boundary condition on its complement. We prove inequalities between the lowest eigenvalues corresponding to two different choices of such boundary conditions on both planar and higher-dimensional domains. We also prove an inequality between higher order mixed eigenvalues and pure Dirichlet eigenvalues on multidimensional polyhedral domains.
\end{abstract}

\section{Introduction}

Let $\Omega$ be a bounded, connected Lipschitz domain in $\rd$, $d \ge 2$, and let $V: \om \to \mathbb{R}$ be a measurable and bounded potential. We consider the eigenvalue problem for the Schr\"odinger operator $-\Delta_\ga +V$ subject to a Dirichlet boundary condition on a non-empty, relatively open subset $\ga$ of the boundary $\bnd$ and a Neumann boundary condition on its complement $\gac$, that is,
\begin{align}
\label{eq:ev problem intro}
\begin{split}
 \begin{cases}
 \hspace*{1.6mm} - \Delta u + Vu = \lambda u & ~~\text{in}~\Omega, \\
 \hspace{7.4mm} \phantom{\,\,\,+ Vu} u = 0 & ~~\text{on}~\Gamma, \\
 \hspace*{3.4mm} \phantom{\,\,\,\,+ Vu} \partial_\nu u = 0 & ~~\text{on}~\Gamma^c,
 \end{cases}
 \end{split}
\end{align}
where $\nu$ denotes the unit normal vector field defined almost everywhere on $\bnd$ pointing outwards. The operator $-\Delta_\ga +V$, including the Laplacian $-\Delta_\ga$ corresponding to the case $V=0$, can be defined via the corresponding quadratic form; it is self-adjoint in $L^2(\om)$, semibounded below by the infimum of $V$ and has a purely discrete spectrum, see Section \ref{sec:preliminaries} for more details. The problem \eqref{eq:ev problem intro} thus admits a discrete sequence of eigenvalues which is bounded below by the infimum of $V$ and which we denote by 
\begin{equation*}
\lambda_1^\ga(V) < \lambda_2 ^\ga(V) \le \lambda_3^\ga(V) \le \ldots
\end{equation*}
counted according to their multiplicities. If $\ga$ coincides with the whole boundary \eqref{eq:ev problem intro} is a pure Dirichlet problem; we denote the resulting eigenvalues by $\lambda_k(V)$ for $k \in \mathbb{N}$. It follows from the variational characterization of the eigenvalues that the inequalities
\begin{equation}
\label{eq:trivial estimates}
\lambda_k^{\ga'}(V) \le \lambda_k^\ga(V) \le \lambda_k(V), \quad k \in \mathbb{N},
\end{equation}
hold for $\ga' \subset \ga \subseteq \bnd$. The first inequality is strict if $\ga \setminus \ga'$ has non-trivial interior and can be proved analogously to \cite[Proposition 2.3]{LR17};
the second follows directly from the variational principles, cf. \eqref{eq:minmax} and \eqref{eq:minmax pure}. The aim of the present article is to establish inequalities for the eigenvalues $\lambda_k^\ga(V)$ which separately improve both trivial estimates of \eqref{eq:trivial estimates}; our results can thus be divided in two different types depending on whether the estimate depends on another mixed eigenvalue or a pure Dirichlet eigenvalue.

In the first type, we limit our analysis to the lowest eigenvalues of Schr\"odinger operators and compare the eigenvalues $\lambda_1^\ga(V)$ of $-\Delta_\ga + V$ with the eigenvalues $\lambda_1^{\ga'}(V)$ of $-\Delta_{\ga'} + V$ corresponding to a different choice of $\ga'$ on the same domain $\om$. We know from  \eqref{eq:trivial estimates} that $\lambda_1^{\ga'}(V)$ is smaller than $\lambda_1^{\ga}(V)$ if $\ga'$ is contained in $\ga$, but this is not always the case if $\ga$ and $\ga'$ do not satisfy an inclusion and $\ga'$ is smaller than $\ga$, that is, $\lambda_1^{\ga}(V)$ does not depend monotonously on the size of $\ga$, cf. the counterexample \cite[Remark 3.3]{Siu16} for the eigenvalues of the Laplacian. Some results comparing $\lambda_1^{\ga'}(V)$ and $\lambda_1^\ga(V)$ for $\ga$ and $\ga'$ disjoint have been established in the case $V=0$ of the Laplacian on planar domains, cf. \cite{AR23,  AR24, Siu16}; to the best of our knowledge, the case of the Schr\"odinger operator, and even the case of the Laplacian in higher dimensions, has not received attention. In this paper we extend the results of \cite{AR23} to Schr\"odinger operators on planar and higher-dimensional domains. Namely, we prove the inequality
\begin{equation}
\label{eq:lowest evs intro}
\lambda_1^{\ga'}(V) \le \lambda_1^\ga(V)
\end{equation}
on convex domains $\om \subset \rd$, $d \ge 2$, if $\ga'$ is either a straight line segment or the subset of an hyperplane and $\overline{\ga} \cup \overline{\ga'} = \bnd$ (Corollary \ref{cor:schrodinger2 gac} and Theorem \ref{thm:higherdim}, respectively), and for $d=2$ in the case in which $\bnd \setminus (\overline{\ga} \cup \overline{\ga'})$ has non-trivial interior (Theorem \ref{thm:schrodinger2 gaprime}); for $d=2$ we prove that \eqref{eq:lowest evs intro} is always strict. In all three results we regulate the potential $V$ in the direction normal to $\ga'$; if $\bnd \setminus (\overline{\ga} \cup \overline{\ga'})$ has non-trivial interior we impose additional conditions on its geometry and on the behaviour of the potential along it. In particular, the geometric conditions on $\bnd$ imply that $\ga'$ is smaller than $\ga$. We wish to point out that, to the best of our knowledge, the higher-dimensional result constitutes a novelty even for eigenvalues of the Laplacian.

As for the second type of result, in Theorem \ref{thm:higherorder evs} we prove the inequality
\begin{equation}
\label{eq:higherdim intro}
\lambda_{k+m}^{\Gamma}(V) \le \lambda_k(V), \quad k \in \mathbb{N},
\end{equation}
between mixed Dirichlet-Neumann eigenvalues and Dirichlet eigenvalues of the Schr\"odinger operator on bounded, convex, polyhedral domains $\om \subset \rd$, $d \ge 2$; this inequality can be regarded as a unification of \cite[Theorem 4.2]{R21} and \cite[Theorem 4.1]{LR17}. We require the potential $V$ to be constant in some directions; $m$ is the number of these directions which are in addition tangential to the Dirichlet portion $\ga$ of the boundary. Inequality \eqref{eq:higherdim intro} exhibits some dimension dependence; if for instance $\ga$ is one face of $\om$ and $V$ a potential which only depends on the direction normal to $\ga$, \eqref{eq:higherdim intro} implies
\begin{equation*}
\lambda_{k+d-1}^{\Gamma}(V) \le \lambda_k(V), \quad k \in \mathbb{N}.
\end{equation*}

The proofs are variational and rely on choosing an appropriate linear combination of eigenfunctions or of their partial derivatives as test function; for the inequalities of the type \eqref{eq:lowest evs intro} we choose an appropriate directional derivative of an eigenfunction for $\lambda_1^\ga(V)$, while for \eqref{eq:higherdim intro} we choose a linear combination of eigenfunctions for $\lambda_1(V), \ldots, \lambda_k(V)$ and of their partial derivatives. In order to estimate the Rayleigh quotient of these test functions we will make use of three different integral identities for the second partial derivatives of Sobolev functions, one of which is proved in this paper, cf. Lemma \ref{lem:GrisvardDimRed}.

This article is organized as follows. In Section \ref{sec:preliminaries} we provide some preliminaries. In Section \ref{sec:planar} and Section \ref{sec:higher dim} we prove inequalities of the type \eqref{eq:lowest evs intro} on planar and higher-dimensional domains respectively. Section \ref{sec:kth ev} is devoted to the proof of inequality \eqref{eq:higherdim intro}.

\section{Preliminaries}
\label{sec:preliminaries}

In this section we fix some notation and present some preliminary results.

\subsection{Schr\"odinger operators and function spaces}

Throughout the whole paper, $\Omega \subset \mathbb{R}^d$, $d \ge 2$, is a bounded, connected Lipschitz domain, see e.g. \cite[Definition 3.28]{M00}; we will make the additional assumption that the boundary $\bnd$ is piecewise smooth. By Rademacher's theorem, for almost all $x \in \bnd$ there exists a uniquely defined exterior unit normal vector $\nu(x)$; $\bnd$ is equipped with the standard surface measure denoted by $\sigma$, cf. \cite{M00}. Note that $\om$ being a Lipschitz domain entails that $\bnd$ does not contain any cusp. 

We denote by $H^s(\om)$, $s > 0$, the $L^2$-based Sobolev space of order $s$ on $\om$; on the boundary we will make use of the Sobolev space $H^{1/2}(\bnd)$ and its dual space $H^{-1/2}(\bnd)$. Recall that there exists a unique bounded, everywhere defined, surjective trace map from $H^1(\om)$ onto $H^{1/2}(\bnd)$ which continuously extends the mapping
\begin{equation*}
 C^\infty(\overline{\om}) \ni u \mapsto u |_{\bnd};
\end{equation*}
we write $u|_{\bnd}$ for the trace of a function $u \in H^1(\om)$. Moreover, for $u \in H^1 (\Omega)$ satisfying $\Delta u\in L^2(\Omega)$ in the distributional sense we define the normal derivative $\partial_\nu u |_{\partial \Omega}$ of $u$ at $\partial \Omega$ to be the unique element in $H^{- 1/2} (\partial \Omega)$ which satisfies the first Green identity
\begin{align}
\label{eq:green}
 \int_\Omega \nabla u \cdot \nabla \overline{v} + \int_\Omega (\Delta u) \overline v = (\partial_\nu u |_{\partial \Omega}, v |_{\partial \Omega} )_{\partial \Omega}, \quad v \in H^1 (\Omega),
\end{align}
where $( \cdot, \cdot)_{\partial \Omega}$ denotes the sesquilinear duality between $H^{1/2} (\partial \Omega)$ and $H^{- 1/2} (\partial \Omega)$. For sufficiently regular $u$, e.g., $u \in H^2 (\Omega)$, the weakly defined normal derivative $\partial_\nu u |_{\partial \Omega}$ coincides with $\nu \cdot \nabla u |_{\partial \Omega}$ almost everywhere on $\partial \Omega$; in this case the duality in \eqref{eq:green} may be replaced by the boundary integral of $\nu \cdot \nabla u |_{\partial \Omega} \overline v |_{\partial \Omega}$ with respect to $\sigma$. 

We now briefly recall the definition of the Schr\"odinger operator with mixed boundary conditions. To this purpose, for any non-empty, relatively open set $\Sigma \subset \bnd$ we denote by $H_{0, \Sigma}^1 (\Omega)$ the Sobolev space
\begin{equation*}
 H_{0, \Sigma}^1 (\Omega) = \left\{ u \in H^1(\om): u |_{\Sigma}=0 \right\},
\end{equation*}
where $u|_\Sigma$ denotes the restriction of the trace $u|_{\bnd}$ to $\Sigma$; if $\Sigma = \bnd$ we write
\begin{equation*}
 H_0^1 (\Omega) = \left\{ u \in H^1(\om): u |_{\bnd}=0 \right\}.
\end{equation*}
We say that a distribution $\psi \in H^{-1/2}(\bnd)$ vanishes on $\Sigma$, and write $\psi|_\Sigma = 0$, if
\begin{equation*}
 \left (\psi, u|_{\bnd} \right)_{\bnd} = 0
\end{equation*}
holds for all $u \in H_{0, \partial \Omega \setminus \overline{\Sigma}}^1 (\Omega)$. For $\ga$ relatively open, non-empty subset of $\bnd$ and $V:\om \to \mathbb{R}$ measurable and bounded we define the Schr\"odinger operator $-\Delta_\ga + V$ subject to a Dirichlet boundary condition on $\ga$ and a Neumann boundary condition on $\gac$ as
\begin{equation*}
(-\Delta_\ga + V)u = -\Delta u + Vu
\end{equation*}
with domain
\begin{equation*}
\dom(-\Delta_\ga + V) = \{ u \in H^1_{0,\ga}(\om) : -\Delta u + Vu \in L^2(\om),\, \partial_\nu u|_{\gac} = 0\}
\end{equation*}
where $V$ acts as a multiplication operator; note that if $V=0$  the Schr\"odinger operator reduces to the (negative) Laplacian with mixed Dirichlet-Neumann boundary conditions. The operator $-\Delta_\ga + V$ corresponds in the sense of \cite[Chapter VI, Theorem 2.1]{Kato} to the quadratic form
\begin{equation}
\label{eq:quadrform schrodinger}
\hmix \ni u \mapsto \int_{\om} |\nabla u|^2 + V |u|^2
\end{equation}
which is semibounded below by the infimum of $V$ and closed; $-\Delta_\ga + V$ is thus self-adjoint in $L^2(\om)$ and its spectrum consists of a discrete sequence of eigenvalues semibounded below by the infimum of $V$ with finite multiplicities converging to $+\infty$. We will make use of the representation of the eigenvalues in terms of the min-max principle
\begin{equation}
\label{eq:minmax}
\lambda_k^\ga(V) = \min_{\substack{L \subset \hmix \\ \dim L=k}} \max_{u \in L \setminus \{0\}} \frac{\int_{\om} (|\nabla u|^2 + V |u|^2)}{\int_{\om} |u|^2}, \quad k \in \mathbb{N};
\end{equation}
in particular the lowest eigenvalue $\lambda_1^\ga(V)$ is non-degenerate, simple and can be expressed by 
\begin{align}
\label{eq:variational}
\lambda_1^\ga(V) =\min_{\substack{u \in \hmix \\ u \neq 0}} \frac{\int_{\Omega} (|\nabla u|^2 + V |u|^2)}{\int_{\Omega} |u|^2}
\end{align}
where $u \in \hmix \setminus \{0\}$ is an eigenfunction of $-\Delta_\ga + V$ corresponding to $\lambda_1^\ga(V)$ if and only if it minimizes \eqref{eq:variational}. As a consequence of \eqref{eq:variational} and of the fact that the mixed Dirichlet-Neumann problem for the Laplacian $-\Delta_\ga$ has positive eigenvalues as soon as $\ga$ is non-empty we get that the lower bound of the infimum of $V$ for the eigenvalues for the eigenvalues $\lambda_k^\ga(V)$ is actually strict, as follows.

\begin{lemma}
\label{lemma:lowest ev strict}
Let $\Omega \subset \mathbb{R}^d$, $d \ge 2$, be a bounded, connected Lipschitz domain, $\ga \subset \bnd$ relatively open and non-empty, and $V: \om \to \mathbb{R}$ measurable and bounded. Then 
\begin{equation*}
\lambda_1^\ga(V) >\inf_{x \in \om} V(x).
\end{equation*}
\end{lemma}

\begin{proof}
Since $\lambda_1^\ga(V)  \ge \inf_{x \in \om} V(x)$ we only need to prove $\lambda_1^\ga(V) \neq \inf_{x \in \om} V(x)$. Let $v \in \hmix$ be a minimizer of \eqref{eq:minmax}, that is,
\begin{equation}
\label{eq:estimate lemma ev}
\int_{\Omega} |\nabla v|^2 + V |v|^2 = \lambda_1^\ga(V) \int_{\Omega} |v|^2.
\end{equation}
holds. On the other hand, the eigenvalue $\lambda_1^\ga$ of the Laplacian corresponding to the same boundary conditions can be computed by setting $V=0$ in \eqref{eq:minmax}, from which we get
\begin{equation*}
\int_{\Omega} |\nabla v|^2  \ge \lambda_1^\ga \int_{\Omega} |v|^2
\end{equation*}
since $v \in \hmix$. We are then able to estimate the right-hand side of \eqref{eq:estimate lemma ev} as follows,
\begin{equation*}
\lambda_1^\ga(V) \int_{\Omega} |v|^2 \ge \lambda_1^\ga \int_{\Omega} |v|^2 + \inf_{x \in \om} V(x) \int_{\Omega} |v|^2,
\end{equation*}
that is,
\begin{equation*}
\lambda_1^\ga(V) \ge \lambda_1^\ga + \inf_{x \in \om} V(x),
\end{equation*}
from which it follows immediately that $\lambda_1^\ga(V) \neq \inf_{x \in \om} V(x)$ as $\lambda_1^\ga > 0$ if $\ga$ is non-empty.
\end{proof}

It follows immediately from Lemma \ref{lemma:lowest ev strict} that if $V$ takes only non-negative values then $\lambda_1^\ga(V) > 0$. The Schr\"odinger operator with a (pure) Dirichlet boundary condition, denoted by $-\Delta_D + V$, is defined analogously as
\begin{equation*}
(-\Delta_D + V)u = -\Delta u + Vu, \quad \dom(-\Delta_D + V) = \{ u \in H^1_0(\om) : -\Delta u + Vu \in L^2(\om) \};
\end{equation*}
it is also self-adjoint in $L^2(\om)$, and its spectrum consists of a discrete sequence of eigenvalues semibounded below by the infimum of $V$ with finite multiplicities converging to $+\infty$. These eigenvalues can be expressed in terms of the min-max principle 
\begin{equation}
\label{eq:minmax pure}
\lambda_k(V) = \min_{\substack{L \subset H^1_0(\om) \\ \dim L=k}} \max_{u \in L \setminus \{0\}} \frac{\int_{\om} (|\nabla u|^2 + V |u|^2)}{\int_{\om} |u|^2}, \quad k \in \mathbb{N}.
\end{equation}

\subsection{Some useful statements}

We conclude the preliminaries by collecting a few statements on the functions in the domain of Schr\"odinger operators subject to mixed Dirichlet-Neumann boundary conditions. We start with the following observation on eigenfunctions, which is a simple consequence of a unique continuation principle. It will be used to show that the eigenvalue inequalities of Section \ref{sec:planar} are always strict; the proof is analogous to \cite[Lemma 4.2]{AR23}.

\begin{lemma}
\label{lem:continuation principle}
Let $\Omega \subset \rd$ be a bounded, connected Lipschitz domain. Let $V:\om \to \R$ measurable and bounded, let $\lambda \in \mathbb{R}$ and let $u \in H^1(\Omega)$ be such that $-\Delta u + V u= \lambda u$ holds in the distributional sense. If $\Lambda \subset \partial \Omega$ is a relatively open, non-empty subset such that $u|_{\Lambda}=0$ and $\partial_{\nu} u|_{\Lambda}=0$ then $u=0$ identically on $\Omega$.
\end{lemma}

Next, we present two sufficient criteria for functions in the domain $\dom (- \Delta_\Gamma+V)$ of $- \Delta_\Gamma+V$ to belong to the Sobolev space $H^2 (\Omega)$. Such regularity holds under assumptions on the angles at the corners at which the transition between Dirichlet and Neumann boundary conditions takes place. The first statement concerns planar domains and is rather well-known; we refer to e.g. \cite[Theorem 2.3.7]{G92} for a proof. 

\begin{proposition}
\label{prop:regularity}
Assume that $\Omega \subset \R^2$ is a bounded Lipschitz domain with piecewise smooth boundary and that $\Gamma, \Gamma^c \subset \partial \Omega$ are relatively open such that $\overline{\Gamma} \cup \overline{\Gamma^c} = \partial \Omega$ holds. Moreover, assume that all angles at which  $\Gamma$ and $\Gamma^c$ meet are less or equal $\pi/2$ and that the angles at all interior corners of $\bnd$ are less or equal $\pi$.  Let $V:\om \to \R$ measurable and bounded. Then 
\begin{align*}
 \dom{(-\Delta_\Gamma + V)} \subset  H^2(\om).
\end{align*}
\end{proposition}

The second statement holds for more general domains of dimension $d \ge 2$ and features the additional assumption that the subset of the boundary which is subject to a Neumann boundary condition is flat, i.e. it is a subset of a hyperplane; the proof relies on a reflection argument with respect to this portion of the boundary. Note that for planar domains this statement is a special case of Proposition \ref{prop:regularity}.

\begin{proposition}
\label{prop:regularity reflection}
Let $\Omega \subset \R^d$, $d \ge 2$, be a bounded Lipschitz domain with piecewise smooth boundary and let $\Gamma, \Gamma^c \subset \partial \Omega$ relatively open such that $\overline{\Gamma} \cup \overline{\Gamma^c} = \partial \Omega$ holds and such that $\ga^c$ is a subset of a $d$-dimensional hyperplane. Moreover, assume that all angles at which $\Gamma$ and $\Gamma^c$ meet are less or equal $\pi/2$ and that the angles at all interior corners of $\bnd$ are less or equal $\pi$. Let $V:\om \to \R$ measurable and bounded. Then 
\begin{align*}
 \dom{(-\Delta_\Gamma + V)} \subset  H^2(\om).
\end{align*}
\end{proposition}

\begin{proof}
Consider the domain $\tilde{\om}$ defined as the union of $\om$ and its reflection over the hyperplane of $\ga^c$, and the potential $\tilde{V}$ on $\tilde{\om}$ defined as the even reflection of $V$ over the same hyperplane; note that the assumptions on the angles of $\bnd$ imply that the angles at all interior corners of $\partial \tilde{\om}$ are less or equal $\pi$. Let now $u$ be any eigenfunction of $-\Delta_{\Gamma} + V$ on $\om$. Since $\tilde{\om}$ is obtained by reflecting over the subset of the boundary $\gac$ which is subject to a Neumann boundary condition, the even reflection $\tilde{u}$ of $u$ defined on $\tilde{\om}$ is an eigenfunction of the pure Dirichlet problem for the operator $-\Delta + \tilde{V}$ on the reflected domain $\tilde{\om}$. By \cite[Proposition 4.8]{AGMT10}, $\tilde{u} \in H^2(\tilde{\om})$ as the angles at all interior corners of $\partial \tilde{\om}$ are less or equal $\pi$; it then follows immediately that $u \in H^2(\om)$.
\end{proof}

\section{An inequality between the lowest mixed eigenvalues of Schr\"odinger operators on planar domains}
\label{sec:planar}

In this section we compare mixed Dirichlet-Neumann eigenvalues corresponding to different configurations of boundary conditions on planar domains. By doing so we are effectively extending to Schr\"odinger operators the work started in \cite{AR23}; in Section \ref{sec:higher dim} we will do the same on domains of dimension higher than two. We first note that the inequalities proved in \cite[Theorem 3.1 and Theorem 3.3]{AR23} extend to any Schr\"odinger operator with a constant potential $V_0$
\begin{equation*}
\lambda^{\ga'}_1(V_0) < \lambda^{\ga}_1(V_0)
\end{equation*}
as adding a constant simply shifts the eigenvalues by that constant. It follows immediately by a perturbation argument that the same inequality holds for potentials which are sufficiently close to constants, namely, if $V \in L^\infty(\om)$ is any real-valued potential then there exists $\tau_0 > 0$ such that 
\begin{equation*}
\lambda^{\ga'}_1(V_0 + \tau V) < \lambda^{\ga}_1(V_0 + \tau V)
\end{equation*}
holds for all $\tau \in \mathbb{R}$ with $|\tau|<\tau_0$. This is true without any additional assumption on $V$. However, these observations have limited relevance as they depend on the strength of the potential; in this section we establish inequalities which do not depend on the strength of $V$, but instead assume $V$ to be constant, or concave and monotonic, along a certain direction.

We start by comparing the lowest eigenvalues on a planar domain $\om$ corresponding to two different configurations: in one, a Dirichlet boundary condition is imposed on an open subset $\ga$ of the boundary and a Neumann boundary condition on its complement $\gac$; in the other, a Dirichlet boundary condition is imposed on an open straight line segment $\ga'$ which is disjoint with $\ga$, but such that the remaining part of the boundary $\bnd \setminus (\overline{\ga} \cup \overline{\ga'})$ has non-trivial interior, and a Neumann boundary condition on its complement $(\ga')^c$. To achieve this we impose the geometric assumption \eqref{eq:altmiraculous} below on $\bnd \setminus \Gamma$, as well as regulate the concavity and the behaviour of the potential $V$ in the unique direction perpendicular to $\ga'$ on various portions of the domain. 

Note that for planar domains $\om$ the boundary $\bnd$ being piecewise smooth implies that $\bnd$ consists of finitely many $C^\infty$-smooth arcs. For almost all $x \in \bnd$ there exists a well-defined outer unit normal vector $\nu(x)$ and a unit tangent vector $\tau(x)$ in the direction of the positive orientation of the boundary satisfying $\nu^\perp(x) = \tau(x)$ for almost all $x \in \bnd$; both vector fields $\tau$ and $\nu$ are piecewise smooth with a finite number of jump discontinuities corresponding to the corners, that is, the intersections between consecutive smooth arcs.

\begin{theorem}
\label{thm:schrodinger2 gaprime}
Let $\om \subset \R^2$ be a bounded, convex Lipschitz domain with piecewise smooth boundary. Let $\Gamma, \Gamma' \subset \partial \Omega$ be disjoint, relatively  open, non-empty sets such that $\Gamma'$ is a straight line segment and $\Gamma$ is connected. Assume that the interior angles of $\partial \Omega$ at both end points of $\Gamma$ are less than $\pi/2$, and that $V \in W^{2,\infty}(\om)$ is real-valued. Let $b$ denote the constant outer unit normal vector of $\Gamma'$, and assume that
\begin{equation}
\label{eq:altmiraculous}
\text{the function}~(\lambda_1^\ga(V) - V)(b \cdot \tau)(b \cdot \nu)~\text{is non-increasing along}~\partial \Omega \setminus \Gamma
\end{equation}
according to positive orientation. Assume in addition that 
\begin{itemize}
\item[(i)] either $b \cdot \nabla V =0$ identically on $\om$ ($V$ is constant along the direction orthogonal to $\ga'$);
\item[(ii)] or $V$ is concave and $(b \cdot \nabla V)(b \cdot \nu)|_{\bnd \setminus \overline{\ga} } \ge 0$.
\end{itemize}
Then 
\begin{equation*}
\lambda^{\ga'}_1(V) < \lambda^{\ga}_1(V).
\end{equation*}
\end{theorem}

In order to facilitate the proof of Theorem \ref{thm:schrodinger2 gaprime} and interpret its assumptions we introduce some notation for $\bnd$. The domain $\Omega$ is convex with a piecewise smooth boundary and $\Gamma \subset \bnd$ is connected, therefore $\partial \Omega \setminus \Gamma$ consists, except for the corners, of finitely many relatively open smooth arcs $\Sigma_1, \dots, \Sigma_N$, which we can enumerate following the boundary in positive orientation from one end point of $\Gamma$ to the other. We denote the corner points of $\partial \Omega$, i.e.\ the points where two consecutive smooth pieces of $\partial \Omega$ meet, by $P_0, P_1, \dots, P_N$ so that $P_0$ and $P_N$ are the end points of $\Gamma$, and $P_{j-1}$ and $P_j$ are the end points of $\Sigma_j$, see Figure~\ref{fig:notationHelp}. 
\begin{figure}[h]
\begin{tikzpicture}[scale=0.4]
\pgfsetlinewidth{0.8pt}
\node[circle,fill=black,inner sep=0pt,minimum size=2pt] (P1) at (7,6.5) {};
\node[circle,fill=black,inner sep=0pt,minimum size=2pt,label=below:{\small {$P_1$}}] (P2) at (5,6) {};
\node[circle,fill=black,inner sep=0pt,minimum size=2pt,label=right:{\small {$P_2$}}] (P3) at (3,4) {};
\node[circle,fill=black,inner sep=0pt,minimum size=2pt,label=right:{\small {$P_3$}}] (P4) at (3,0) {};
\node[circle,fill=black,inner sep=0pt,minimum size=2pt] (P5) at (4.22,-1.64) {};
\node[white] (G) at (5.26,-2.36) {}; 
\node[white] (H) at (6.32,-2.76) {}; 
\node[white] (J) at (7.5,-2.98) {}; 
\node[circle,fill=black,inner sep=0pt,minimum size=2pt,label=above left:{\small {$P_5$}}] (P6) at (8,-3) {};
\begin{pgfonlayer}{bg} 
\pgfsetlinewidth{0.8pt}
\path (P1) (P2) (P3);
\draw[gray] (P2) to[bend right=25] (P3); 
\node at (6,7) {$\Sigma_1$};
\draw[gray](P1.center) to (P2.center);
\node at (3.2,5.7) {$\Sigma_2$};
\node at (6.5,5.6) {\small $P_0$};
\node at (4.7,-1.2) {\small $P_4$};
\node at (9,2) {$\ga$};
\draw (P3.center) to node[black][left]{$\Gamma'\!=\!\Sigma_3$} (P4.center);
\draw[gray](P4.center) to node[black][below left]{$\Sigma_4$} (P5.center);
\draw[gray] (P5.center) to node[black][left]{} (P6.center);
\draw(6.32,-2.76) node[black][below]{$\Sigma_5$};
\path (P1) (P6);
\draw[black] (P1) to[bend left=10] (P6); 
\end{pgfonlayer}
\node at (barycentric cs:P1=1,P3=1,P4=1,P6=1) {$\Omega$};
\end{tikzpicture}
\caption{A convex domain with piecewise smooth boundary.}
\label{fig:notationHelp}
\end{figure}
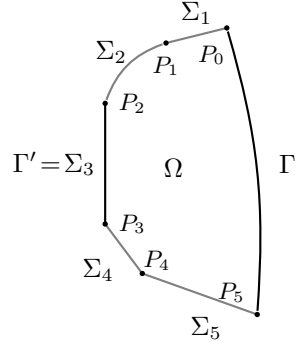
The function $(b \cdot \tau)(b \cdot \nu)$ on $\partial \Omega \setminus \Gamma$ is piecewise smooth with only a finite number of jump discontinuities corresponding to the corners, as the vector fields $\tau$ and $\nu$. At a corner $P_j$, where two smooth arcs $\Sigma_j$ and $\Sigma_{j+1}$ meet, the condition \eqref{eq:altmiraculous} must be read
\begin{align*}
 \lim_{\Sigma_{j+1} \ni x \to P_j} & ( \lambda_1^\ga(V)-V(x))(b \cdot \tau (x))(b \cdot \nu (x)) \le \\
 &\lim_{\Sigma_{j} \ni x \to P_j}( \lambda_1^\ga(V)-V(x))(b \cdot \tau (x))(b \cdot \nu (x)).
\end{align*}
By construction, on the straight arc $\Gamma'$,
\begin{equation*}
 (b \cdot \tau)(b \cdot \nu)|_{\Gamma'}=0
\end{equation*}
identically as $b$ is normal to $\Gamma'$. Therefore, condition \eqref{eq:altmiraculous} is automatically satisfied on $\ga'$. Note also that, since on $\ga'$ $\nu=b$ identically, on $\ga'$ condition (ii) reduces to $b \cdot \nabla V \ge 0$, while on $\bnd \setminus (\overline{\ga} \cup \overline{\ga'})$ the sign of $b \cdot \nu$ is not fixed. The following two examples demonstrate domains and potentials for which Theorem \ref{thm:schrodinger2 gaprime} holds; in the first one the potential satisfies condition (i), while in the second condition (ii).

\begin{example}
\label{example:trapezium}
Let $\om \subset \mathbb{R}^2$ be a trapezium, i.e. a quadrilateral with at least one pair of parallel sides, called bases, and assume that the two angles adjacent to its longer base are acute. We choose $\ga'$ to be the shorter base and $\ga$ the longer base, and denote the remaining sides by $\Sigma_1$ and $\Sigma_3$ following the notation introduced above, see Figure \ref{fig:trapezium triangle}. Assume without loss of generality that $\om$ is rotated such that $\ga$ and $\ga'$ are parallel to the $x_2$-axis and in particular $\ga' = \{ 0 \} \times [a,b]$. Then $b=(-1,0)^\top$, $(b \cdot \tau)(b \cdot \nu) = \tau_1 \tau_2$ and in particular $(b \cdot \tau)(b \cdot \nu)|_{\Sigma_1} > 0$, $(b \cdot \tau)(b \cdot \nu)|_{\Sigma_3} < 0$. Let $\phi \in C^\infty_c([a,b])$ be non-negative and define $V$ as follows;
\begin{equation*}
V(x_1,x_2)=
\begin{cases}
\phi(x_2), \quad (x_1,x_2) \in \om : x_2 \in [a,b]; \\
0, \,\,\,\,\quad \quad \mathrm{elsewhere \,\, in} \om.
\end{cases}
\end{equation*}
Then $V \in W^{2,\infty}(\om)$ and $b \cdot \nabla V = -\partial_1 V = 0$ on $\om$. Also, $V \ge 0$, from which it follows that $\lambda_1^\ga(V)>0$ by Lemma \ref{lemma:lowest ev strict}, and in particular $V|_{\Sigma_1 \cup \Sigma_3}=0$. The function $(\lambda_1^\ga(V) - V)(b \cdot \tau)(b \cdot \nu)$ from condition \eqref{eq:altmiraculous} then equals $\lambda_1^\ga(V) (b \cdot \tau)(b \cdot \nu)$ on $\Sigma_1$ and $\Sigma_3$ and vanishes identically on $\ga'$ as observed above. Conditions \eqref{eq:altmiraculous} and (i) of Theorem \ref{thm:schrodinger2 gaprime} are then verified for this choice of $\om$ and $V$, and $\lambda^{\ga'}_1(V) < \lambda^{\ga}_1(V)$ holds.
\end{example}

\begin{example}
Let $\om \subset \mathbb{R}^2$ be an obtuse triangle, i.e. a triangle with one angle bigger than $\pi/2$, and choose $\ga$ and $\ga'$ to be two of its sides such that the angle which they enclose is strictly less than $\pi/2$. We assume without loss of generality that the remaining side is situated between the end point of $\ga'$ and the starting point of $\ga$ according to positive orientation of the boundary; we name this side $\Sigma_1$ according to the notation introduced above, see Figure \ref{fig:trapezium triangle}. As in the previous example we assume without loss of generality that $\om$ is rotated such that $\ga'$ is parallel to the $x_2$-axis and in particular $\ga' = \{0\} \times [0,a]$, i.e. the end point shared by $\ga'$ and $\Sigma_1$ lies in the origin. Then $(b \cdot \tau)(b \cdot \nu) = \tau_1 \tau_2$, from which $(b \cdot \tau)(b \cdot \nu)|_{\Sigma_1} < 0$, and $(b \cdot \nu)|_{\Sigma_1} > 0$ as the angle between $b$ and $\nu|_{\Sigma_1}$ is smaller than $\pi/2$. Since $\nu|_{\Sigma_1}$ is a constant vector, we define the potential $V$ on $\om$ as 
\begin{equation*}
V(x) = a(x \cdot \nu|_{\Sigma_1}), \quad x \in \om,
\end{equation*}
where $a > 0$ is small enough such that $\lambda_1^\ga(V) > 0$; note that such $a$ exists by an argument similar to the one performed in the proof of Lemma \ref{lemma:lowest ev strict}. Then, $V \in W^{2,\infty}(\om)$ and $V$ is linear, hence concave. Further, $V$ vanishes identically on $\Sigma_1$ as any point $x \in \Sigma_1$ can be identified with a vector which is perpendicular to $ \nu|_{\Sigma_1}$, and $b \cdot \nabla V = b \cdot \nu|_{\Sigma_1}$ holds on $\om$. Condition \eqref{eq:altmiraculous} holds as $(\lambda_1^\ga(V)-V) (b \cdot \tau)(b \cdot \nu)$ vanishes identically on $\ga'$ and equals $\lambda_1^\ga(V)(b \cdot \tau)(b \cdot \nu) < 0$ on $\Sigma_1$; while condition (ii) is verified since $V$ is concave, $(b \cdot \nabla V)(b \cdot \nu)|_{\ga'}=(b \cdot \nu)|_{\Sigma_1} |b|^2> 0$, and $(b \cdot \nabla V)(b \cdot \nu)|_{\Sigma_1} = (b \cdot \nu)^2|_{\Sigma_1} > 0$. Conditions \eqref{eq:altmiraculous} and (ii) of Theorem \ref{thm:schrodinger2 gaprime} are then verified for this choice of $\om$ and $V$, and $\lambda^{\ga'}_1(V) < \lambda^{\ga}_1(V)$ holds.
\end{example}

 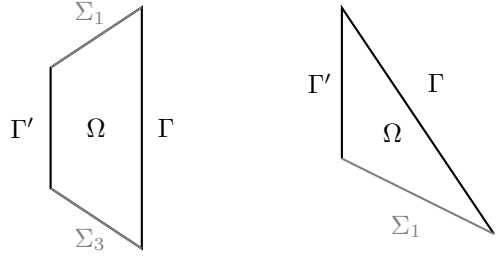
\begin{figure}[h]
\begin{minipage}[c][4cm][c]{0,3\textwidth}
\begin{tikzpicture}[scale=0.4]
\pgfsetlinewidth{0.8pt}
\coordinate (A) at (0, 0);
\coordinate (B) at (0, 4);
\coordinate (C) at (3, 6);
\coordinate (D) at (3, -2);
\draw (A) -- (B) -- (C) -- (D) -- cycle;
\draw[gray](B.center) to (C.center);
\draw[gray](D.center) to (A.center);
\coordinate (M) at (-0.2, 2);
\coordinate (N) at (3.2, 2);
\node[left] at (M) {$\ga'$};
\node[right] at (N) {$\ga$};
\node at (barycentric cs:A=1,B=1,C=1,D=1) {$\Omega$};
\node[gray] at (1.3,5.8) {$\Sigma_1$};
\node[gray] at (1.3,-1.7) {$\Sigma_3$};
\end{tikzpicture}
\end{minipage}
\begin{minipage}[c][4cm][c]{0,3\textwidth}
\begin{tikzpicture}[scale=0.5]
\pgfsetlinewidth{0.8pt}
\coordinate (O) at (0, 0);
\coordinate (A) at (0, 4);
\coordinate (B) at (4,-2);
\draw (O) -- (A) -- (B);
\draw[gray](O.center) to (B.center);
\coordinate (M) at (0,2);
\coordinate (N) at (2, 2);
\node[left] at (M) {$\ga'$};
\node[right] at (N) {$\ga$};
\node at (barycentric cs:O=1,A=1,B=1) {$\Omega$};
\node[gray] at (1.7,-1.8) {$\Sigma_1$};
\end{tikzpicture}
\end{minipage}
\caption{Theorem \ref{thm:schrodinger2 gaprime} applies to these polygons.}
\label{fig:trapezium triangle}
\end{figure}

We now prove Theorem \ref{thm:schrodinger2 gaprime} inspired by the proof of \cite[Theorem 3.3]{AR23}. We will make use of the following integration by parts result; we refer to \cite[Lemma 4.3]{AR23} for the proof. In order to do so we introduce the signed curvature of $\partial \Omega$ with respect to the outer unit normal $\nu$, defined at each point of $\partial \Omega$ except at the corners, see for instance \cite[Section 2.2]{PR10}. It can be expressed as 
\begin{equation*}
\kappa = \tau' \cdot \nu,
\end{equation*}
where $\tau$ is the unit tangent vector field along the boundary and the derivative $\tau'$ is to be understood piecewise via an arclength parametrization of the boundary in positive direction. On convex domains $\kappa(x) \le 0$ holds for almost all $x \in \partial \Omega$.

\begin{lemma}
\label{lem:Grisvard}
Assume that $\Omega \subset \mathbb{R}^2$ is a bounded Lipschitz domain with piecewise smooth boundary consisting, except for the corners, of finitely many smooth arcs $\Gamma_1, \dots, \Gamma_N$, and let $\kappa$ denote the curvature of $\partial \Omega$ w.r.t.\ the unit normal pointing outwards. Then 
\begin{align*}
\int_{\Omega} (\partial_1^2 u)(\partial_2^2 u)= \int_{\Omega} (\partial_1 \partial_2 u)^2 - \frac{1}{2} \int_{\partial \Omega} \kappa \abs{\nabla u}^2 \, \textup{d} \sigma 
\end{align*}
holds for all real-valued $u \in V^2 (\Omega)$, where
\begin{align*}
 V^2 (\Omega) = \Big\{ w \in H^2 (\Omega) : \text{on each}~\Sigma_j,~w |_{\Sigma_j} = 0~\text{or}~\partial_\nu w |_{\Sigma_j} = 0 \Big\},
\end{align*}
i.e.\ for all functions in $H^2 (\Omega)$ which satisfy a Dirichlet or Neumann boundary condition on each smooth arc.
\end{lemma}

\begin{proof}[Proof of Theorem \ref{thm:schrodinger2 gaprime}.]
Let $\Gamma' \subset \partial \Omega$ be an open straight line segment and $\Gamma \subset \partial \Omega$ be relatively open and connected such that $\Gamma \cap \Gamma' = \emptyset$ and the interior angles of $\partial \Omega$ at both end points of $\Gamma$ are strictly less than $\pi/2$. We denote by $b$ the constant outer unit normal vector of $\ga'$. Let $u$ be a real-valued eigenfunction of $-\Delta_{\Gamma}+V$ corresponding to the eigenvalue $\lambda_1^{\Gamma}(V)$, and let $v = b \cdot \nabla u$, the directional derivative of $u$ in the direction of $b$. Proposition \ref{prop:regularity} yields that $v \in H^1(\Omega)$ as the angles adjacent to $\Gamma$ are strictly smaller than $\pi/2$ and $\Omega$ is convex. Moreover, the Neumann boundary condition which $u$ satisfies on $\ga'$ implies $v |_{\ga'} = 0$, i.e.\ $v \in H_{0, \ga'}^1 (\Omega)$. Note that $v$ is non-trivial since $b \cdot \nabla u=0$ identically on $\Omega$ together with $u=0$ on $\Gamma$ would imply $u=0$ on $\Omega$. This is due the fact that, due to the angle requirement and convexity of $\Omega$, $\Gamma$ is not a straight line segment orthogonal to $\Gamma'$.

Our first aim is to prove that 
\begin{align}
\label{eq:almost}
 \frac{\int_\Omega |\nabla v|^2 + V|v|^2 }{\int_\Omega |v|^2} \leq \lambda_1^\Gamma(V)
\end{align}
which combined with the variational principle \eqref{eq:minmax} yields $\lambda_1^{\ga'}(V) \le \lambda_1^\ga(V)$; we will prove strictness of the inequality separately. First, integration by parts yields
\begin{align}
\label{eq:half}
\begin{split}
\lambda_1^\Gamma(V) \int_\Omega |v|^2 & = \lambda_1^\Gamma(V)  \int_\Omega \nabla u \cdot b b^\top \nabla u  \\
& = \lambda_1^\Gamma(V)  \left( - \int_\Omega u \diver \left(b b^\top \nabla u \right) + \int_{\partial \Omega} u b b^\top \nabla u \cdot \nu \, \textup{d} \sigma \right) \\
& = \int_\Omega \Delta u \diver \left(b b^\top \nabla u \right) - \int_\Omega Vu \diver \left(b b^\top \nabla u \right) \\
& \phantom{\,=\,} + \lambda_1^\Gamma(V) \int_{\partial \Omega} u (b \cdot \nabla u)(b \cdot \nu) \, \textup{d} \sigma.
\end{split}
\end{align}
We rewrite the first domain integral on the right-hand side of \eqref{eq:half} using Lemma~\ref{lem:Grisvard},
\begin{align}
\label{eq:otherhalf}
\begin{split}
 \int_\Omega \Delta u \diver \left(b b^\top \nabla u \right) & = \int_\Omega (\Delta u ) \left( b_1^2 \partial_1^2 u + 2 b_1 b_2 \partial_1 \partial_2 u + b_2^2 \partial_2^2 u \right) \\
 & = \int_\Omega b_1^2 |\nabla \partial_1 u|^2 + 2 \int_\Omega b_1 b_2 \nabla \partial_1 u \cdot \nabla \partial_2 u \\
 & \quad + \int_\Omega b_2^2 |\nabla \partial_2 u|^2 - \frac{1}{2} \left(b_1^2 + b_2^2 \right) \int_{\partial \Omega} \kappa |\nabla u|^2 \, \textup{d} \sigma \\
 & = \int_\Omega |\nabla v|^2  - \frac{1}{2}  \int_{\partial \Omega} \kappa |\nabla u|^2 \, \textup{d} \sigma;
\end{split}
\end{align}
while the second domain integral on the right-hand side of \eqref{eq:half} may be integrated by parts,
\begin{align}
\label{eq:middlehalf}
\begin{split}
\int_\Omega Vu \diver \left(b b^\top \nabla u \right) & = - \int_{\om} \nabla(Vu) \cdot  b b^\top \nabla u + \int_\bnd Vu (b \cdot \nabla u)(b \cdot \nu) \,\mathrm{d}\sigma \\
& = - \int_{\om} V (\nabla u \cdot b b^\top \nabla u)  -  \int_{\om} u (\nabla V \cdot b b^\top \nabla u) \\
& \phantom{\,=\,} + \int_\bnd Vu (b \cdot \nabla u)(b \cdot \nu) \,\mathrm{d}\sigma \\
& = - \int_{\om} V |v|^2 - \int_{\om} u \left(b \cdot \nabla V \right) \left( b \cdot \nabla u \right) \\
& \phantom{\,=\,} + \int_\bnd Vu (b \cdot \nabla u)(b \cdot \nu) \,\mathrm{d}\sigma.
\end{split}
\end{align}
By plugging \eqref{eq:otherhalf} and \eqref{eq:middlehalf} into \eqref{eq:half} we obtain 
\begin{align}
\label{eq:fundeq0}
\begin{split}
\lambda_1^\Gamma(V) \int_\Omega |v|^2  & = \int_\Omega (|\nabla v|^2 + V |v|^2) - \frac{1}{2}  \int_{\partial \Omega} \kappa |\nabla u|^2 \, \textup{d} \sigma + \int_{\om} u \left(b \cdot \nabla V \right) \left( b \cdot \nabla u \right)  \\
& \phantom{\,=\,} + \lambda_1^\Gamma(V) \int_{\partial \Omega} u (b \cdot \nabla u)(b \cdot \nu) \, \textup{d} \sigma - \int_\bnd Vu (b \cdot \nabla u)(b \cdot \nu) \,\mathrm{d}\sigma.
\end{split}
\end{align}
The convexity of $\Omega$ implies $\kappa \leq 0$ almost everywhere on $\partial \Omega$, therefore we obtain
\begin{align}
\label{eq:fundineq0}
\begin{split}
\lambda_1^\Gamma(V) \int_\Omega |v|^2  & \ge  \int_\Omega (|\nabla v|^2 + V |v|^2) + \int_{\om} u \left(b \cdot \nabla V \right) \left( b \cdot \nabla u \right) \\
& \phantom{\,=\,} + \int_{\partial \Omega} ( \lambda_1^\Gamma(V) u - V u)(b \cdot \nabla u)(b \cdot \nu) \, \textup{d} \sigma
\end{split}
\end{align}
and inequality \eqref{eq:almost} follows if we can show that
\begin{align}
\label{eq:done}
\int_ {\om} u (b \cdot \nabla V)(b \cdot \nabla u) + \int_{\partial \Omega} (\lambda_1^\ga(V) u - Vu) (b \cdot \nabla u)(b \cdot \nu) \, \textup{d} \sigma \geq 0;
\end{align}
we prove that each summand is non-negative. 

We start with the domain integral. If condition (i) holds, $b \cdot \nabla V=0$ holds almost everywhere on $\om$ and the integral vanishes. Else, if condition (ii) holds, integration by parts yields
\begin{align}
\label{eq:step 3 bV comp}
\begin{split}
\int_{\om} u \left(b \cdot \nabla V \right) \left( b \cdot \nabla u \right) & = \int_{\om} u \nabla V \cdot b b^\top \nabla u = \frac{1}{2} \int_{\om} \nabla V \cdot b b^\top \nabla (u^2) \\
& = -\frac{1}{2} \int_{\om} \diver (b b^\top \nabla V) u^2 + \frac{1}{2} \int_{\bnd} \left( b b^\top \nabla V \cdot \nu \right) u^2 \, \mathrm{d}\sigma \\
& = -\frac{1}{2} \int_{\om} b^\top H_V b u^2 +\frac{1}{2} \int_{\bnd} (b \cdot \nabla V)(b \cdot \nu) u^2\, \mathrm{d}\sigma \\
& = -\frac{1}{2} \int_{\om} b^\top H_V b u^2 + \frac{1}{2} \int_{\bnd \setminus \overline{\ga}} (b \cdot \nabla V)(b \cdot \nu) u^2\, \mathrm{d}\sigma
\end{split}
\end{align}
where in the last step we used $u|_{\ga}=0$. The domain integral on the right-hand side is non-negative since $V$ is concave and thus the associated Hessian matrix $H_V$ is non-positive; while the integral over $\bnd \setminus \overline{\ga'}$ is also non-negative since $(b \cdot \nabla V)(b \cdot \nu) \ge 0$ on $\bnd \setminus \overline{\ga}$ by assumption. We have thus proved that 
\begin{equation*}
\int_{\om} u \left(b \cdot \nabla V \right) \left( b \cdot \nabla u \right) \ge 0.
\end{equation*}

We now focus on the boundary integral in \eqref{eq:done}. Note first that the integrand vanishes on $\overline{\ga} \cup \overline{\ga'}$ as $u |_\Gamma = 0$ and $b \cdot \nabla u|_{\ga'}=0$ constantly. To estimate the integral over the remainder of the boundary $\bnd \setminus (\overline{\ga} \cup \overline{\ga'})$ we will make use of the notation introduced above. Recall that $\partial \Omega \setminus \Gamma$ consists, except for the corners, of finitely many smooth arcs $\Sigma_1, \dots, \Sigma_N$ enumerated following the boundary in positive orientation from one end point of $\Gamma$ to the other, and that we denote the corner points adjacent to each $\Sigma_j$ by $P_j$ and $P_{j-1}$, see Figure \ref{fig:notationHelp}. Moreover, we define the functions
\begin{equation*}
 t_j  \coloneqq (b \cdot \tau)(b \cdot \nu)|_{\Sigma_j}, \quad j = 1, \dots, N;
\end{equation*}
note that for the index $j$ for which $\Sigma_j$ equals $\Gamma'$ we have $t_j = 0$ identically as $b$ is normal to $\Gamma'$, and that $t_j$ is constant if $\Sigma_j$ is a straight line segment.

At almost each point of $\partial \Omega \setminus \Gamma$ we can express the vector $b$ as a linear combination of $\tau$ and $\nu$,
\begin{equation*}
b = (b \cdot \tau) \tau + (b \cdot \nu) \nu
\end{equation*}
so that on $\partial \Omega \setminus \Gamma$, by the Neumann boundary condition imposed on $u$,
\begin{align*}
b \cdot \nabla u = (b \cdot \tau) \, \tau \cdot \nabla u + (b \cdot \nu) \, \nu \cdot \nabla u = (b \cdot \tau) \, \partial_\tau u + (b \cdot \nu) \, \partial_\nu u = (b \cdot \tau) \,\partial_\tau u
\end{align*}
holds. Inserting this expression into the integrand of the boundary integral in \eqref{eq:done} and integrating over any arc~$\Sigma_j$, $j = 1, \ldots, N$, we obtain
\begin{align}
\label{eq:firstpartbtaubnu}
\begin{split}
& \int_{\Sigma_j} (\lambda_1^\ga(V) u - Vu) (b \cdot \nabla u)(b \cdot \nu) \, \textup{d} \sigma \\
& = \int_{\Sigma_j} (b \cdot \tau) (b \cdot \nu)  (\lambda_1^\ga(V) u - Vu) \partial_\tau u \, \textup{d} \sigma \\
& = \lambda_1^\ga(V) \int_{\Sigma_j} t_j u \partial_\tau u \, \textup{d} \sigma - \int_{\Sigma_j} t_j V u \partial_\tau u \, \textup{d} \sigma \\
& = \frac{1}{2} \bigg( \lambda_1^\ga(V) \int_{\Sigma_j} t_j \partial_\tau (u^2) \, \textup{d} \sigma - \int_{\Sigma_j} V t_j \partial_\tau (u^2) \, \textup{d} \sigma \bigg).
\end{split}
\end{align}
where $\partial_\tau$ denotes the derivative in the direction of the tangential vector field $\tau$ of $\bnd$. By the fundamental theorem of calculus we get
\begin{align*}
\int_{\Sigma_j} t_j \partial_\tau (u^2) \, \textup{d} \sigma = - \int_{\Sigma_j} u^2 \partial_\tau t_j \, \textup{d} \sigma + t_j (P_j) u^2 (P_j) - t_j (P_{j-1}) u^2 (P_{j-1})
\end{align*}
and
\begin{align*}
\int_{\Sigma_j} t_j V \partial_\tau (u^2) \, \textup{d} \sigma & = \int_{\Sigma_j} (\partial_\tau (V t_j u^2) - u^2 \partial_\tau(V t_j)) \, \textup{d} \sigma \\
& = - \int_{\Sigma_j} u^2 \partial_\tau (V t_j) \, \textup{d} \sigma \\
& \phantom{\,\,\,=} + V(P_j) t_j (P_j) u^2 (P_j) - V(P_{j-1}) t_j (P_{j-1}) u^2 (P_{j-1}).
\end{align*}
By plugging these two identities into \eqref{eq:firstpartbtaubnu} and using $u = 0$ on $\Gamma$, in particular $u (P_0) = u (P_N) = 0$, we obtain
\begin{align}
\label{eq:decomposed}
\begin{split}
& \int_{\bnd} (\lambda_1^\ga(V) u - Vu) (b \cdot \nabla u)(b \cdot \nu) \, \textup{d} \sigma \\
& = \frac{1}{2} \sum_{j=1}^{N} \bigg( - \lambda_1^\ga(V) \int_{\Sigma_j} u^2 \partial_\tau t_j \, \textup{d} \sigma + \int_{\Sigma_j} u^2 \partial_\tau (V t_j) \, \textup{d} \sigma\\
& \quad +\lambda_1^\ga(V) t_j (P_j) u^2 (P_j) - \lambda_1^\ga(V) t_j (P_{j-1}) u^2 (P_{j-1}) \\
& \quad + V(P_{j-1}) t_j (P_{j-1}) u^2 (P_{j-1}) - V(P_j) t_j (P_j) u^2 (P_j) \bigg) \\
& = \frac{1}{2} \bigg( \sum_{j=1}^{N}  \int_{\Sigma_j} u^2 \left( - \lambda_1^\ga(V) \partial_\tau t_j + \partial_\tau (V t_j) \right) \, \textup{d} \sigma \\
& \quad + \sum_{j=1}^{N - 1} \lambda_1^\ga(V) [t_j - t_{j + 1}] (P_j) u^2 (P_j) \bigg) - \sum_{j=1}^{N - 1} [Vt_j - Vt_{j + 1}] (P_j) u^2 (P_j) \bigg) \\
& = \frac{1}{2} \bigg( \sum_{j=1}^{N}  \int_{\Sigma_j} u^2 \left( - \lambda_1^\ga(V) \partial_\tau t_j + \partial_\tau (V t_j) \right) \, \textup{d} \sigma \\
& \quad + \sum_{j=1}^{N - 1} (\lambda_1^\ga(V) [t_j - t_{j + 1}] - [V t_j - Vt_{j + 1}]) (P_j) u^2 (P_j) \bigg).
\end{split}
\end{align}
We can now conclude that this boundary integral is non-negative. Indeed, by the assumption \eqref{eq:altmiraculous} of the theorem we have
\begin{align*}
\lambda_1^\ga(V) \partial_\tau t_j - \partial_\tau (V t_j) \le 0 \quad \text{on}~\Sigma_j, \quad j = 1, \dots, N,
\end{align*}
and $[V t_j - Vt_{j + 1}] (P_j) - \lambda_1^\ga(V) [t_j - t_{j + 1}] (P_j) \leq 0$ at each corner $P_j$, $j = 1, \dots, N - 1$. This proves
\begin{align*}
\int_{\partial \Omega} (\lambda_1(V)u - Vu) (b \cdot \nabla u)(b \cdot \nu) \, \textup{d} \sigma \geq 0;
\end{align*}
thus \eqref{eq:done} follows and with \eqref{eq:fundineq0} this proves the inequality
\begin{equation*}
\lambda^{\Gamma'}_1(V) \le \lambda^{\Gamma}_1(V).
\end{equation*}

We now argue that this inequality is strict by distinguishing several cases based on the geometry of $\ga$. Assume for a contradiction that equality holds; then \eqref{eq:fundeq0} and the following computations imply that both
\begin{align}
\label{eq:curvatureZero}
 \int_{\partial \Omega} \kappa |\nabla u|^2 \, \textup{d} \sigma = 0,
\end{align}
and
\begin{align}
\label{eq:boundIntZeroV}
\begin{split}
& \sum_{j=1}^{N}  \int_{\Sigma_j} u^2 \left( - \lambda_1^\ga(V) \partial_\tau t_j + \partial_\tau (V t_j) \right) \, \textup{d} \sigma \\
& + \sum_{j=1}^{N - 1} (\lambda_1^\ga(V) [t_j - t_{j + 1}] - [V t_j - Vt_{j + 1}]) (P_j) u^2 (P_j) = 0
\end{split}
\end{align}
hold. 

\textit{Case 1.} In the first case, assume that $\Gamma$ is not a straight line segment, i.e., there exists a relatively open subset $\Lambda \subset \Gamma$ on which $\kappa$ is non-zero, and in particular uniformly negative as $\Omega$ is convex. Then \eqref{eq:curvatureZero} implies $\nabla u = 0$ almost everywhere on $\Lambda$. Thus, both $u$ and $\partial_\nu u$ vanish on $\Lambda$, and Lemma \ref{lem:continuation principle} implies $u = 0$ constantly in $\Omega$, a contradiction. 

\textit{Case 2.} Now, assume that there exists a relatively open subset $\Lambda \subset \bnd \setminus \overline{\ga}$ on which $\kappa$ is non-zero, and in particular $\Lambda \subset \bnd \setminus (\overline{\ga} \cup \overline{\ga'})$ as $\ga'$ is a straight line segment. Then we can argue the same way as in Case 1 to show that $\nabla u = 0$ holds almost everywhere on $\Lambda$. Note that if $\lambda_1^{\Gamma'}(V) = \lambda_1^\Gamma(V)$ holds, then \eqref{eq:almost} yields
\begin{align*}
\lambda_1^\Gamma(V) = \lambda_1^{\Gamma'}(V) = \frac{\int_\Omega ( |\nabla v|^2 + V |v|^2)}{\int_\Omega |v|^2},
\end{align*}
which in turn implies that $v = b \cdot \nabla u$ is an eigenfunction of $- \Delta_{\Gamma'}+V$ corresponding to $\lambda_1^{\Gamma'}(V)$. Thus, $\partial_\nu v=0$ on $\bnd \setminus \overline{\ga'}$ and in particular on $\Lambda$. On the other hand, $v = b \cdot \nabla u = 0$ on $\Lambda$ as $\nabla u = 0$ there; Lemma \ref{lem:continuation principle} then yields $v = 0$ constantly in $\Omega$, a contradiction since $v$ is an eigenfunction.

\textit{Case 3.} We are now in the case in which $\om$ is a polygon. Assume that $\Gamma$ is a straight line segment which is not parallel to $\Gamma'$ and denote by $P$ one of its end points; recall that the angle at $P$ is by assumption strictly less than $\pi/2$. At each point of $\ga$ we can express the vector $b$ normal to $\ga'$ as a linear combination of the unit vectors $\tau$ and $\nu$ which are tangential respectively normal to $\ga$,
\begin{equation}
\label{eq:bLinearComb}
 b = (b \cdot \tau) \, \tau + (b \cdot \nu) \, \nu,
\end{equation}
so that on $\ga$ we obtain
\begin{align}
\label{eq:reint2}
v = b \cdot \nabla u = (b \cdot \tau) \, \tau \cdot \nabla u + (b \cdot \nu) \, \nu \cdot \nabla u = (b \cdot \tau) \, \partial_\tau u + (b \cdot \nu) \, \partial_\nu u.
\end{align}
Since $u$ vanishes on $\ga$, we have $\partial_\tau \partial_\tau u = 0$ and thus 
\begin{equation*}
0 = - \lambda_1^\Gamma(V) u = \Delta u - Vu = \Delta u = \partial_\tau \partial_\tau u + \partial_\nu \partial_\nu u = \partial_\nu \partial_\nu u
\end{equation*}
on $\ga$, which combined with \eqref{eq:reint2} yields
\begin{align}
\label{eq:reint4}
\begin{split}
 \partial_\nu v & =  (b \cdot \tau) \, \partial_\nu \partial_\tau u + (b \cdot \nu) \, \partial_\nu \partial_\nu u =  (b \cdot \tau) \, \partial_\nu \partial_\tau u
\end{split}
\end{align}
on $\ga$. Now, for a contradiction, assume that $\lambda_1^{\Gamma'}(V) = \lambda_1^\Gamma(V)$ holds. Then, as above, \eqref{eq:almost} implies that $v = b \cdot \nabla u$ is an eigenfunction of $- \Delta_{\Gamma'}+V$ corresponding to $\lambda_1^{\Gamma'}(V)$. Then $\partial_\nu v =0$ on $(\Gamma')^c$ and in particular on $\ga$, so that the left-hand side of \eqref{eq:reint4} vanishes on $\ga$. As $\ga$ and $\Gamma'$ cannot be parallel, we have $b \cdot \tau \neq 0$ on $\ga$, and \eqref{eq:reint4} yields that $0 = \partial_\nu \partial_\tau u = \partial_\tau \partial_\nu u$ on $\Gamma$, i.e there exists a constant $c \in \mathbb{R}$ such that on $\Gamma$
\begin{equation}
\label{eq:normalDerConstant}
\partial_\nu u = c.
\end{equation}
Moreover, the corner $P$ is a critical point of $u$ as $u$ satisfies a Neumann boundary condition on $\Gamma^c$ and a Dirichlet boundary condition on $\Gamma$ and the two sides are not perpendicular. Thus $\nabla u(P) =  0$, which combined with \eqref{eq:normalDerConstant} gives $\partial_\nu u = 0$ on $\Gamma$. Since $u=0$ on $\Gamma$, Lemma~\ref{lem:continuation principle} yields $u = 0$ constantly in $\Omega$, a contradiction. 

\textit{Case 4.} Now, assume that $\om$ is a polygon and that $\Gamma$ is a line segment parallel to $\Gamma'$, see for instance Figure \ref{fig:polygon}; in particular $\Gamma$ and $\Gamma'$ are not adjacent to each other.
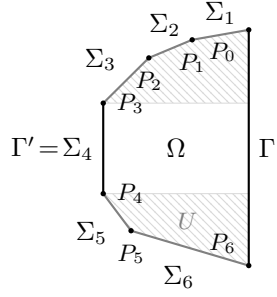
\begin{figure}[h]
\begin{tikzpicture}[scale=0.3]
\pgfsetlinewidth{0.8pt}
\node[circle,fill=black,inner sep=0pt,minimum size=2pt,label=below left:{\small {$P_0$}}] (P0) at (9.38,7.24) {};
\node[circle,fill=black,inner sep=0pt,minimum size=2pt,label=below:{\small {$P_1$}}] (P1) at (6.9,6.8) {};
\node[circle,fill=black,inner sep=0pt,minimum size=2pt,label=below:{\small {$P_2$}}] (P2) at (5,6) {};
\node[circle,fill=black,inner sep=0pt,minimum size=2pt,label=right:{\small {$P_3$}}] (P3) at (3,4) {};
\node[circle,fill=black,inner sep=0pt,minimum size=2pt,label=right:{\small {$P_4$}}] (P4) at (3,0) {};
\node[circle,fill=black,inner sep=0pt,minimum size=2pt,label=below:{\small {$P_5$}}] (P5) at (4.22,-1.64) {};
\node[white] (U1) at (9.4,4) {}; 
\node[white] (U2) at (9.4,0) {}; 
\node[white] (G) at (5.26,-2.36) {}; 
\node[white] (H) at (6.32,-2.76) {}; 
\node[white] (J) at (7.5,-2.98) {}; 
\node[circle,fill=black,inner sep=0pt,minimum size=2pt,label=above left:{\small {$P_6$}}] (P6) at (9.38,-3.18) {};
\begin{pgfonlayer}{bg} 
\pgfsetlinewidth{0.8pt}
\draw[white] (P0.center)--(P1.center)--(P2.center)--(P3.center)--(U1.center);
\draw[ultra thin,gray!30!white] (P3.center)--(U1.center);
\fill[pattern=north west lines, pattern color=gray!30!white] (P0.center)--(P1.center)--(P2.center)--(P3.center)--(U1.center)-- (P0.center);
\draw[white] (P6.center)--(P5.center)--(P4.center)--(U2.center)--(P6.center);
\draw[ultra thin,gray!30!white] (P4.center)--(U2.center);
\fill[pattern=north west lines, pattern color=gray!30!white] (P6.center)--(P5.center)--(P4.center)--(U2.center)--(P6.center);
\draw[gray](P0.center) to (P1.center);
\node at (8.14,8) {$\Sigma_1$};
\draw[gray](P1.center) to (P2.center);
\node at (5.7,7.45) {$\Sigma_2$};
\draw[gray](P2.center) to node[black][above left]{$\Sigma_3$} (P3.center);
\draw (P3.center) to node[black][left]{$\Gamma'\!=\!\Sigma_4$} (P4.center);
\draw[gray](P4.center) to node[black][below left]{$\Sigma_5$} (P5.center);
\draw[gray] (P5.center) to node[black][left]{} (P6.center);
\draw(6.32,-2.76) node[black][below]{$\Sigma_6$};
\draw (P6.center) to node[black][right]{$\Gamma$} (P0.center);
\end{pgfonlayer}
\node at (barycentric cs:P0=1,P3=1,P4=1,P6=1) {$\Omega$};
\node[gray] at (6.7,-1.2) {\small $U$};
\end{tikzpicture}
\caption{A polygonal domain where $\ga'$ is parallel to $\ga$; the striped regions make up the set $U$ defined below.}
\label{fig:polygon}
\end{figure}
Assume in addition that $V(P_j) \neq \lambda_1^\ga(V)$ for some $j = 1,\ldots,N-1$. In particular, assume that $V(P_k) \neq \lambda_1^\ga(V)$ for a fixed $k \in \{1,\ldots,N-1\}$ such that $P_k$ is the end point of $\Gamma'$ ($P_4$ in Figure \ref{fig:polygon}). Then \eqref{eq:boundIntZeroV} holds, and since by the assumption \eqref{eq:altmiraculous} the function $(V-\lambda_1^\ga(V))(b \cdot \tau)(b \cdot \nu)$ is non-decreasing it holds this implies that each summand is equal to zero separately; in particular we have
\begin{align}
\label{eq:contr thirdcase}
[\lambda_1^\ga(V)(t_j - t_{j + 1})- (Vt_j - Vt_{j + 1})] (P_j) u^2 (P_j) = 0, \quad j = 1, \dots, N - 1;
\end{align}
note that $t_k=0$ identically as $\Sigma_k = \ga'$. Due to the convexity of $\Omega$ and the requirements on the angles adjacent to $\Gamma$ in the theorem, the interior angle of $\partial \Omega$ at $P_k$ must be strictly larger than $\pi/2$, and, in particular, $t_k (P_k) - t_{k + 1} (P_k) = - t_{k + 1} (P_k) \neq 0$. Thus, \eqref{eq:contr thirdcase} only holds if $u(P_k)=0$ since we are off the case $V(P_k) = \lambda_1^\ga(V)$. Note that, as argued above, $\lambda_1^{\Gamma'}(V) = \lambda_1^\Gamma(V)$ and \eqref{eq:almost} yield that $v = b \cdot \nabla u$ is an eigenfunction of $- \Delta_{\Gamma'}+V$ corresponding to $\lambda_1^{\Gamma'}(V)$. Its sign can then be chosen to be strictly positive or negative inside $\om$, which implies that $u$ must be either strictly increasing or strictly decreasing inside $\Omega$ in the direction of the vector $b$. However, the straight line parallel to $b$ through $P_k$ intersects $\partial \Omega$, except for $P_k$, at a point in $\Gamma$, where $u$ satisfies a Dirichlet boundary condition. But the function $u$ vanishes at both its intersection points with $\partial \Omega$, a contradiction. Since no segment $\Sigma_j$ is parallel or perpendicular to $\Gamma'$ due to the assumptions on convexity and angles, we can repeat the same argument at each point $P_j$ where $V(P_j) \neq \lambda_1^\ga(V)$ to derive a contradiction.

\textit{Case 5.} In this final case we assume that $\om$ is a polygon, $\Gamma$ is a line segment parallel to $\Gamma'$, and $V(P_1) = \ldots = V(P_{N-1}) = \lambda_1^\ga(V)$ holds; we again refer the reader to Figure \ref{fig:polygon}. Note that since $\om$ is a polygon $\partial_\tau t_j=0$ holds on each $\Sigma_j$, $j=1,\ldots, N$, and \eqref{eq:boundIntZeroV} yields
\begin{equation}
\label{eq:case5 integral}
\int_{\Sigma_j} u^2 t_j \partial_\tau V \, \textup{d} \sigma = 0, \quad j=1,\ldots, N.
\end{equation}
Let now $j \in \{1, \ldots, N\}$ such that $\Sigma_j$ does not coincide with $\Gamma'$; note that on $\Gamma'$ $t_j=0$ holds and \eqref{eq:case5 integral} is satisfied. By the reasoning performed in Case 4 we know that $u$ cannot vanish anywhere on $\Sigma_j$, so by \eqref{eq:case5 integral} $t_j \partial_\tau V = 0$ holds on $\Sigma_j$. However, $t_j \neq 0$ on $\Sigma_j$ since none of the segments $\Sigma_j$ is parallel or perpendicular to $\Gamma'$; by \eqref{eq:case5 integral} $\partial_\tau V = 0$ must then hold on $\Sigma_j$, i.e. $V$ is constant on $\Sigma_j$. Since $V(P_1) = \ldots = V(P_{N-1}) = \lambda_1^\ga(V)$, this yields that
\begin{equation}
\label{eq:case5 constant}
V= \lambda_1^\ga(V) \,\,\, \mathrm{constantly} \,\, \mathrm{on} \,\, \bnd \setminus (\overline{\ga} \cup \overline{\ga'}).
\end{equation}
We now distinguish two cases based on conditions (i) and (ii) of the statement. If condition (ii) holds, $V$ is concave and \eqref{eq:case5 constant} implies that $V(x) \ge  \lambda_1^\ga(V)$ for all $x \in \om$. This means that $\lambda_1^\ga(V) = \inf_{x \in \om} V(x)$, in contradiction with Lemma \ref{lemma:lowest ev strict}. Else, condition (i) yields that $b \cdot \nabla V =0$ on $\om$, which combined with \eqref{eq:case5 constant} implies that $V=\lambda_1^\ga(V)$ holds everywhere on the subset $U$ of $\om$ which is obtained by moving opposite $b$ starting from each point of $\bnd \setminus (\overline{\ga} \cup \overline{\ga'})$, see Figure \ref{fig:polygon}. As above, $v = b \cdot \nabla u$ is an eigenfunction of $- \Delta_{\Gamma'}+V$ corresponding to $\lambda_1^{\Gamma'}(V)$, i.e. $-\Delta v + V v = \lambda_1^\ga(V) v$ holds, and in particular on $U$ 
\begin{equation*}
-\Delta v = -V v + \lambda_1^\ga(V) v = 0.
\end{equation*}
By the same reasoning, $-\Delta u =0$ on $U$. By the Neumann boundary condition, $\partial_\nu u = 0$ and thus $\partial_\tau \partial_\nu u = 0$ holds on each $\Sigma_j$; analogously, $\partial_\nu v = \partial_\nu (b \cdot \nabla u) = \partial_b \partial_\nu u = 0$ holds on each $\Sigma_j$. Therefore, since on each $\Sigma_j$ $\tau$ and $b$ are neither parallel nor perpendicular, we can conclude $\partial_\nu \partial_\nu u = 0$ which combined with $-\Delta u =0$ implies $\partial_\tau \partial_\tau u = 0$ on each $\Sigma_j$. Thus, $\partial_\tau u=c$ on each $\Sigma_j$. In particular, $\nabla u=0$ holds at both end points of $\ga$ by the same reasoning performed in Case 3 and thus $\partial_\tau u= \tau \cdot \nabla u = 0$ on each $\Sigma_j$ sharing an end point with $\ga$ ($\Sigma_1$ and $\Sigma_6$ in Figure \ref{fig:polygon}). This means that $u$ is constant on such segments $\Sigma_j$; however, by the Dirichlet boundary condition on $\ga$, $u$ vanishes at the end points which the segments $\Sigma_j$ share with $\ga$ ($P_0$ and $P_6$ in Figure \ref{fig:polygon}) and thus $u=0$ identically on these $\Sigma_j$. As $\Sigma_j \subset \gac$ for all $j=1, \ldots, N$, this yields a contradiction with the Neumann boundary condition $\partial_\nu u = 0$ on $\gac$ by Lemma \ref{lem:continuation principle}. This concludes the proof of Theorem~\ref{thm:schrodinger2 gaprime}.
\end{proof}

\begin{remark}
(i) For a constant potential $V_0$, $\lambda_1^\ga(V_0) = \lambda_1^\ga + V_0$ where $\lambda_1^\ga$ is the lowest eigenvalue of the Laplacian subject to the same boundary conditions. Then since $\lambda_1^\ga > 0$ condition \eqref{eq:altmiraculous} of Theorem \ref{thm:schrodinger2 gaprime} rewrites as 
\begin{equation*}
\text{the function}~(b \cdot \tau)(b \cdot \nu)~\text{is non-increasing along}~\partial \Omega \setminus \Gamma
\end{equation*}
as in \cite[Theorem 3.3]{AR23}, and conditions (i) and (ii) are automatically satisfied as $\nabla V = 0$. This is in accordance with the above observation that inequalities for eigenvalues of the Laplacian are trivially satisfied when adding a constant potential $V_0$. 
\\
(ii) By Proposition \ref{prop:regularity}, the requirements of Theorem \ref{thm:schrodinger2 gaprime} on the angles of $\bnd$ adjacent to $\ga$ could be relaxed to allow one or both angles to equal $\pi/2$. Indeed, the proof of the non-strict eigenvalue inequality holds as long as $\ga$ and $\ga'$ do not enclose an angle equal to $\pi/2$. However, we wish to point out that, even if we are off this case, the arguments for strict inequality do not hold if $\ga$ is a straight line segment; this can be quickly verified by choosing $\ga$ and $\ga'$ to be opposite sides of any rectangle as by symmetry $\lambda_1^\ga = \lambda_1^{\ga'}$ holds for the eigenvalues of the Laplacian.
\end{remark}

If $\ga$ and $\ga'$ exhaust the whole boundary, i.e. $\bnd \setminus (\overline{\ga} \cup \overline{\ga'}) = \emptyset$, Theorem \ref{thm:schrodinger2 gaprime} simplifies to the following result comparing the lowest eigenvalues of two configurations that are dual to each other, that is, Dirichlet and Neumann boundary conditions are interchanged from one another. 

\begin{corollary}
\label{cor:schrodinger2 gac}
Let $\om \subset \R^2$ be a bounded, convex Lipschitz domain with piecewise smooth boundary. Let $\ga, \ga' \subset \bnd$ be disjoint, relatively open, non-empty sets such that $\ga'$ is a straight line segment and $\overline{\ga} \cup \overline{\ga'}=\bnd$. We denote by $b$ the constant outer unit normal vector of $\ga'$. Assume that the interior angles of $\bnd$ at both end points of $\ga$ are strictly less than $\pi/2$. Assume that $V \in W^{2,\infty}(\om)$ is real-valued and 
\begin{itemize}
\item[(i)] either $b \cdot \nabla V =0$ ($V$ is constant along the direction orthogonal to $\ga'$);
\item[(ii)] or $V$ is concave and $b \cdot \nabla V|_{\ga'} \ge 0$.
\end{itemize}
Then
\begin{equation*}
\lambda^{\ga'}_1(V) < \lambda^{\ga}_1(V).
\end{equation*}
\end{corollary}

\begin{proof}
If $\ga, \ga' \subset \bnd$ are such that $\overline{\ga} \cup \overline{\ga'}=\bnd$, then $\bnd \setminus \ga = \overline{\ga'}$. Condition \eqref{eq:altmiraculous} of Theorem \ref{thm:schrodinger2 gaprime} is then automatically satisfied on $\ga'$ as $(b \cdot \tau)(b \cdot \nu)$ vanishes on $\ga'$, while condition (ii) reduces to $b \cdot \nabla V =0$ on $\ga'$ as $b \cdot \nu = \abs{b}^2 > 0$ on $\ga'$. The statement then follows immediately from Theorem \ref{thm:schrodinger2 gaprime}.
\end{proof}

Two examples of domains to which Corollary \ref{cor:schrodinger2 gac} applies are shown in Figure \ref{fig:two domains}; since the angles adjacent to the longest side of any triangle are always smaller than $\pi/2$, Corollary \ref{cor:schrodinger2 gac} applies in particular to any triangle if we choose $\ga'$ to be its longest side and $\ga$ the union of the two remaining sides.
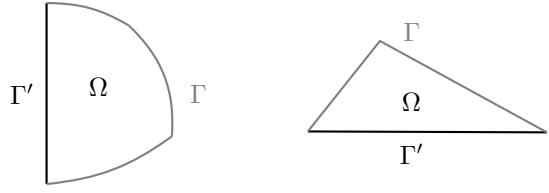
\begin{figure}[h]
\begin{minipage}[c][4cm][c]{0,3\textwidth}
\begin{tikzpicture}[scale=0.4]
\pgfsetlinewidth{0.8pt}
\node (A) at (1,1) {};
\node (B) at (1,7) {};
\node (C) at (3.7, 6.26) {};
\node (D) at (5.12, 2.58) {};
\draw (A.center) to (B.center);
\path (1,1) coordinate (A) (1,7) coordinate (B) (3.7, 6.26) coordinate (C) (5.12, 2.58) coordinate (D);
\draw[gray]  (B) to[bend left=15] (C); 
\draw[gray] (C) to[bend left=25] (D); 
\draw[gray]  (D) to[bend left=15] (A); 
\node at (0.2,4) {$\ga'$};
\node[gray] at (6,4) {$\ga$};
\node at (barycentric cs:A=1,B=1,C=1,D=1) {$\Omega$};
\end{tikzpicture}
\end{minipage}
\begin{minipage}[c][4cm][c]{0,3\textwidth}
\begin{tikzpicture}[scale=0.6]
\pgfsetlinewidth{0.8pt}
\node (G) at (1.42,3) {};
\node (H) at (6.66,2.98) {};
\node (I) at (3.,5) {};
\draw (G.center) to (H.center);
\draw[gray] (H.center) to (I.center) to (G.center);
\node  at (3.7, 2.5) {$\ga'$};
\node[gray]  at (3.7, 5.2) {$\ga$};
\node at (barycentric cs:G=1,H=1,I=1) {$\Omega$};
\end{tikzpicture}
\end{minipage}
\caption{Two domains to which Corollary \ref{cor:schrodinger2 gac} applies.}
\label{fig:two domains}
\end{figure}

We now collect several examples of potentials satisfying the assumptions of Corollary \ref{cor:schrodinger2 gac}; in all examples we assume $\om \subset \mathbb{R}^2$ to be a domain satisfying the assumptions of Corollary \ref{cor:schrodinger2 gac} and denote by $b=(b_1,b_2)^\top$ the normal vector to $\ga'$.

\begin{example}
\label{example:first}
Consider the potential
\begin{equation*}
V(x) = \alpha e^{\beta (x_1 +x_2)}, \quad x=(x_1, x_2)^\top \in \om,
\end{equation*}
for $\alpha>0$ and $\beta \in \mathbb{R} \setminus \{ 0 \}$. Then all partial derivatives of $V$ of order $n$ equal $\beta^n V$; it follows immediately that $V \in W^{2,\infty}(\om)$ and that $V$ is concave as $\alpha > 0$. Further, we have $b \cdot \nabla V = \beta V (b_1 + b_2)$ and condition (i) of Corollary \ref{cor:schrodinger2 gac} is satisfied if $b_1 = - b_2$, while condition (ii) holds if $\beta$ and $b_1+b_2$ have the same sign.
\end{example}

\begin{example}
Let $m = (m_1, m_2)^\top \in \mathbb{R}^2$, $q \in \mathbb{R}$ and 
\begin{equation*}
V(x) = m\cdot x + q = \binom{m_1}{m_2} \cdot \binom{x_1}{x_2} + q, \quad x=(x_1, x_2)^\top  \in \om.
\end{equation*}
Then $V \in W^{2,\infty}(\om)$, $V$ is both convex and concave, and $b \cdot \nabla V = b \cdot m \geq 0$, condition (ii) of Corollary \ref{cor:schrodinger2 gac}, holds if the vectors $b$ and $m$ form an angle smaller or equal $\pi/2$.
\end{example}

\begin{example}
Let
\begin{equation*}
V(x) = g(x_2), \quad x=(x_1, x_2)^\top  \in \om,
\end{equation*}
where $g \in W^{2,\infty}(\om)$ is real-valued and such that $g'$ is not identically zero. Then $b \cdot \nabla V = -\partial_{x_1} V = 0$ and condition (i) of Corollary \ref{cor:schrodinger2 gac} is satisfied.
\end{example}

\begin{example}
\label{example:last}
Consider the potential
\begin{equation*}
V(x) = \dist (x, \ga), \quad x \in \om.
\end{equation*}
Then $V \in W^{2,\infty}(\om)$ and $b \cdot \nabla V \ge 0$ as $b$ points in the direction opposite to $\ga$ and the distance from $\ga$ increases proceeding in the direction of $b$. To see that $V$ is concave, we consider the domain $\tilde{\om}$ obtained by reflecting $\om$ over the straight line segment $\ga'$ and define $\tilde{V}(x) = \dist (x, \bnd)$ for all $x \in \tilde{\om}$; note that $\tilde{\om}$ is convex as $\om$ is convex and the angles adjacent to $\ga'$ are smaller than $\pi/2$. Then $\tilde{V}$ is concave on $\tilde{\om}$ and so is its restriction to $\om$, which is precisely $V$. Then $V$ satisfies condition (ii) of Corollary \ref{cor:schrodinger2 gac}.
\end{example}

\section{A higher-dimensional inequality between the lowest mixed eigenvalues of Schr\"odinger operators}
\label{sec:higher dim}

In this section we extend Corollary \ref{cor:schrodinger2 gac} to dimensions higher than two, that is, we compare the lowest eigenvalues of the Schr\"odinger operator corresponding to two configurations where Dirichlet and Neumann boundary conditions are interchanged on domains with dimension $d \ge 3$. By doing so we are effectively comparing the lowest eigenvalues of the Laplacian corresponding to the same configurations.

\begin{theorem}
\label{thm:higherdim}
Let $\om \subseteq \rd$, $d \ge 3$, be a bounded, convex Lipschitz domain with piecewise smooth boundary. Let $\ga, \ga' \subset \bnd$ be disjoint, relatively open, non-empty sets such that $\ga'$ is a subset of a $d$-dimensional hyperplane and $\overline{\ga} \cup \overline{\ga'}=\bnd$. We denote by $b$ the constant outer unit normal vector of $\ga'$. Assume that the interior angles of $\bnd$ adjacent to $\ga$ are less or equal $\pi/2$. Then 
\begin{equation}
\label{eq:laplacian higherdim ineq}
\lambda_1^{\ga'} \le \lambda_1^\ga.
\end{equation}
Assume in addition that $V \in W^{2,\infty}(\om)$ is real-valued and 
\begin{itemize}
\item[\rm{(i)}] either $b \cdot \nabla V =0$ ($V$ is constant along the direction orthogonal to $\ga'$);
\item[\rm{(ii)}] or $V$ is concave and $b \cdot \nabla V|_{\ga'} \ge 0$.
\end{itemize}
Then
\begin{equation}
\label{eq:schrodinger higherdim ineq}
\lambda_1^{\ga'}(V) \le \lambda_1^\ga(V).
\end{equation}
\end{theorem}

Some examples of three-dimensional domains to which Theorem \ref{thm:higherdim} applies are collected in Figure \ref{fig:solids}; by Proposition \ref{prop:regularity reflection}, Theorem \ref{thm:higherdim} holds for any $d$-dimensional domain which can be reflected over the hyperplane of $\ga'$ to obtain a convex domain. Examples of potentials satisfying the assumptions of Theorem \ref{thm:higherdim} can be easily obtained by generalizing Examples \ref{example:first} -- \ref{example:last} to dimensions higher than two.
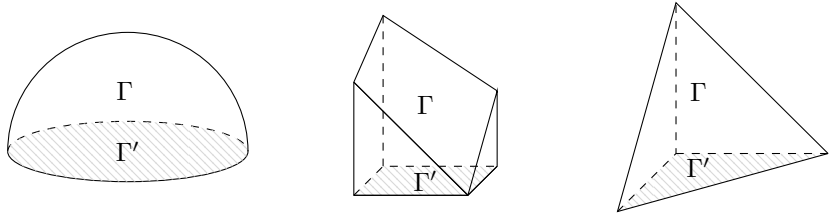
\begin{figure}[h]
\begin{minipage}[c][5cm][c]{0,3\textwidth}
\centering
\begin{tikzpicture}[scale=0.4]
 \draw (-3.95,.5) arc (180:360:3.95 and 1);
 \draw [dashed] (-3.95,.5) arc (180:0:3.95 and 1);
\fill[pattern=north west lines, pattern color=gray!30!white] (3.95,.5) arc (0:360:3.95 and 1);
 \draw (-3.95,.5) arc (180:0:3.95 and 3.95);
 \draw (0,0.5) node () {$\ga'$};
 \draw (-0.1,2.5) node () {$\ga$};
\end{tikzpicture}
\end{minipage}
\begin{minipage}[c][5cm][c]{0,3\textwidth}
\centering
\begin{tikzpicture}[scale=0.5]
\coordinate[] (O) at (0,0,0);
\coordinate[] (A2) at (0,4,0);
\coordinate[] (B1) at (3,0,0);
\coordinate[] (B2) at (0,0,2);
\coordinate[] (B3) at (3,0,2);
\coordinate[] (C2) at (0,3,2);
\coordinate[] (C3) at (3,2,0);
\fill[pattern=north west lines, pattern color=gray!30!white](B2)--(B3)--(B1)--(O);
\draw (B2)--(B3)--(B1);
\draw [dashed] (B2)--(O)--(B1);
\draw [dashed] (A2)--(O);
\draw (B3)--(B1)--(C3)--(B3);
\draw (C2)--(B2)--(B3);
\draw (B3)--(C2)--(A2)--(C3);
\draw (C2)--(B3);
\draw (1.5,0,1) node () {$\ga'$};
\draw (1.5,2,1) node () {$\ga$};
\end{tikzpicture}
\end{minipage}
\begin{minipage}[c][5cm][c]{0,3\textwidth}
\centering
\begin{tikzpicture}[scale=0.4]
\coordinate[] (O) at (0,0,0);
\coordinate[] (A) at (5,0,0);
\coordinate[] (B) at (0,0,5);
\coordinate[] (C) at (0,5,0);
\fill[pattern=north west lines, pattern color=gray!30!white](A)--(B)--(O)--(A);
\draw (A)--(B)--(C)--(A);
\draw [dashed] (B)--(O)--(C);
\draw [dashed] (A)--(O);
\draw (1.2,0,1.2) node () {$\ga'$};
\draw (1.2,2.5,1.2) node () {$\ga$};
\end{tikzpicture}
\end{minipage}
\caption{For these domains, Theorem \ref{thm:higherdim} holds; note that $\ga'$ could be chosen to be any of the faces of the third domain.}
\label{fig:solids}
\end{figure}

The proof of Theorem \ref{thm:higherorder evs} makes use of the following estimate which can be regarded as an extension of Lemma \ref{lem:Grisvard} to higher dimensions; its proof uses a dimension reduction trick from the proof of \cite[Th\'eor\`eme~2.1]{G75} and hinges on the fact that the sign of the signed curvature $\kappa$ introduced in Section \ref{sec:planar} is non-positive on convex domains. 

\begin{lemma}
\label{lem:GrisvardDimRed}
Let $\om \subset \rd$, $d \ge 2$, be bounded, convex Lipschitz domain with piecewise smooth boundary. Let $\ga, \ga' \subset \bnd$ be disjoint, relatively open, non-empty sets such that $\ga'$ is a subset of a $d$-dimensional hyperplane and $\overline{\ga} \cup \overline{\ga'}=\bnd$. Let $\om$ be rotated such that $\ga'$ is perpendicular to the first coordinate axis. Then
\begin{align*}
\int_{\Omega} (\partial_1 \partial_j u)^2 \le \int_{\Omega} (\partial_1^2 u)(\partial_j^2 u)
\end{align*}
holds for all $j \in \{1,\ldots, d\}$ and all real-valued $u \in U^2 (\Omega)$, where
\begin{align*}
 U^2 (\Omega) = \Big\{ w \in H^2 (\Omega) : w |_{\ga} = 0~\text{and}~\partial_\nu w |_{\ga'} = 0 \Big\},
\end{align*}
i.e.\ for all functions in $H^2 (\Omega)$ which satisfy a Dirichlet boundary condition on $\ga$ and a Neumann boundary condition on $\ga'$.
\end{lemma}

\begin{proof}
For $d=2$ the assertion follows immediately from Lemma \ref{lem:Grisvard} as $\kappa(x)\le 0$ holds for almost all $x \in \bnd$ for convex $\om$.

Let now $d \ge 3$. We assume $j \neq 1$ since otherwise the claim is satisfied trivially. Assume $j=2$. For a fixed $\boldsymbol{z} \in \mathbb{R}^{d-2}$ we define the two-dimensional set
\begin{equation*}
\om_{\boldsymbol{z}} = \{(x_1, x_2, \boldsymbol{z})^\top : (x_1, x_2, \boldsymbol{z})^\top \in \om \}
\end{equation*}
as the intersection of $\om$ with the plane $(0, 0, \boldsymbol{z})^\top + \spann \{\ee_1, \ee_2 \}$. Then $\om_{\boldsymbol{z}}$ is a bounded convex Lipschitz domain, $u|_{\om_{\boldsymbol{z}} } \in H^2(\om_{\boldsymbol{z}}) \cap V^2(\om_{\boldsymbol{z}})$ for almost all $\boldsymbol{z} \in \mathbb{R}^{d-2}$, and $u|_{\om_{\boldsymbol{z}} }$ satisfies a Dirichlet boundary condition on $\partial \om_{\boldsymbol{z}} \cap \, \ga$ and a Neumann boundary condition on $\partial \om_{\boldsymbol{z}} \cap \, \ga'$ as $\nu|_{\ga'} = \ee_1$ constantly up to sign. We can then apply Lemma \ref{lem:Grisvard} to $u|_{\om_{\boldsymbol{z}} }$ on $\om_{\boldsymbol{z}}$ to obtain 
\begin{align*}
\int_{\Omega_{\boldsymbol{z}}} (\partial_1^2 u)(\partial_2^2 u)= \int_{\Omega_{\boldsymbol{z}}} (\partial_1 \partial_2 u)^2 - \frac{1}{2} \int_{\partial \Omega_{\boldsymbol{z}}} \kappa \abs{\nabla u}^2 \, \textup{d} \sigma 
\end{align*}
for almost all $\boldsymbol{z} \in \mathbb{R}^{d-2}$. Now, we integrate over $\boldsymbol{z}$ to obtain 
\begin{align*}
\int_{\om} (\partial_1 \partial_2 u)^2 &  = \int_{\boldsymbol{z} \in \mathbb{R}^{d-2}} \int_{\om_{\boldsymbol{z}} } (\partial_1 \partial_2 u)^2 \\
& = \int_{\boldsymbol{z} \in \mathbb{R}^{d-2}} \int_{\om_{\boldsymbol{z}}}  (\partial_1^2 u)(\partial_2^2 u) + \frac{1}{2} \int_{\boldsymbol{z} \in \mathbb{R}^{d-2}} \int_{\partial \om_{\boldsymbol{z}}} \kappa \abs{\nabla u}^2 \, \textup{d} \sigma \\
& = \int_{\om}  (\partial_1^2 u)(\partial_2^2 u) + \frac{1}{2} \int_{\boldsymbol{z} \in \mathbb{R}^{d-2}} \int_{\partial \om_{\boldsymbol{z}}} \kappa \abs{\nabla u}^2 \, \textup{d} \sigma
\end{align*}
from which
\begin{align*}
\int_{\om} (\partial_1 \partial_2 u)^2 \le \int_{\om}  (\partial_1^2 u)(\partial_2^2 u)
\end{align*}
follows immediately as $\om_{\boldsymbol{z}}$ is convex and hence $\kappa$ is non-positive almost everywhere on $\partial \om_{\boldsymbol{z}}$. The remaining cases $j=3,\ldots,d$ follow analogously by intersecting $\om$ with an appropriate choice of 2-dimensional plane parallel to $\spann \{\ee_1,\ee_j\}$. 
\end{proof}

\begin{proof}[Proof of Theorem \ref{thm:higherdim}]
We prove inequality \eqref{eq:schrodinger higherdim ineq}; inequality \eqref{eq:laplacian higherdim ineq} follows immediately by choosing $V=0$ identically on $\om$. 

Let $\ga' \subset \bnd$ be an open subset of a $d$-dimensional hyperplane and $\Gamma \subset \partial \Omega$ be relatively open and connected such that $\Gamma \cap \Gamma' = \emptyset$ and the first eigenfunction of $-\Delta_\ga$ belongs to $H^2(\om)$. Without loss of generality we can assume that $\Omega$ is rotated such that $\ga'$ is perpendicular to the first coordinate axis. We then have $b = \ee_1$ identically up to sign; in what follows we assume that $b = \ee_1$. Let $u$ be a real-valued eigenfunction of $-\Delta_\Gamma+V$ corresponding to the eigenvalue $\lvg{1}$, and consider $b \cdot \nabla u = \partial_1 u$, the first partial derivative of $u$. Then, $\partial_1 u \in H^1_{0,\ga'}(\om)$ by Proposition \ref{prop:regularity reflection} and by the Neumann boundary condition $\partial_1 u|_{\Gamma'}=0$; note that $\partial_1 u$ is non-trivial as $\partial_1 u=0$ identically on $\om$ together with $u=0$ on $\ga$ would imply $u=0$ on $\om$. Our aim is to prove that 
\begin{align}
\label{eq:almost1d}
\frac{\int_\Omega |\nabla \partial_1 u|^2 + V|v|^2}{\int_\Omega |\partial_1 u|^2} & \leq \lambda_1^\Gamma(V)
\end{align}
which combined with the variational principle \eqref{eq:minmax} yields $\lambda_1^{\ga'}(V) \le \lambda_1^\ga(V)$. We start by using integration by parts ($\mathbf{0}$ denotes the ($d-1$)-dimensional vector whose components are all 0),
\begin{align}
\label{eq:multidstart}
\begin{split}
\lvg{1} \int_{\Omega} (\partial_{1}u)^2 = & \lvg{1} \int_{\Omega} \begin{pmatrix} \partial_1 u \\ \mathbf{0} \end{pmatrix} \cdot \nabla u \\
& = \lvg{1} \left( - \int_{\Omega} \Div \begin{pmatrix} \partial_1 u  \\ \mathbf{0} \end{pmatrix} u + \int_{\bnd} u \begin{pmatrix} \partial_1 u \\ \mathbf{0} \end{pmatrix} \cdot \nu \,\mathrm{d}\sigma \right) \\
& = - \int_{\Omega} (\partial_{11} u)(- \Delta u + Vu) \\
& = \int_{\Omega} \partial_{11} u \Delta u - \int_{\Omega} \partial_{11} u Vu.
\end{split}
\end{align}
The boundary integral vanishes as $u$ vanishes on $\ga$ and the normal derivative $\partial_\nu u= b \cdot \nabla u = \partial_1 u$ vanishes on $\ga'$, where $\nu=\ee_1$ constantly. We then estimate the first integral on the right-hand side of \eqref{eq:multidstart} using Lemma \ref{lem:GrisvardDimRed},
\begin{align}
\label{eq:auxmultid1}
\begin{split}
\int_{\Omega} \partial_{11} u \Delta u = \int_{\Omega} \sum_{\substack{j=1}}^d(\partial_{11} u)(\partial_{jj} u)  \ge \int_{\Omega}  \sum_{j=1}^d (\partial_{1j} u)^2 = \int_{\Omega} \abs{\nabla (\partial_1u)}^2.
\end{split}
\end{align}
As for the second integral on the right-hand side of \eqref{eq:multidstart}, we use integration by parts,
\begin{align}
\label{eq:auxmultid2}
\begin{split}
\int_{\Omega} \partial_{11} u Vu & = \int_{\Omega} \Div \begin{pmatrix} \partial_1 u \\ \mathbf{0} \end{pmatrix} Vu \\
& = - \int_{\Omega} \begin{pmatrix} \partial_1 u \\ \mathbf{0} \end{pmatrix} \cdot \nabla (Vu) + \int_{\bnd} Vu \begin{pmatrix} \partial_1 u \\ \mathbf{0} \end{pmatrix} \cdot \nu \,\mathrm{d}\sigma  \\
& = - \int_{\Omega} (\partial_1 u)(\partial_1 (Vu)) \\
& = - \int_{\Omega} V (\partial_1 u)^2 - \int_{\Omega} u (\partial_1 u)(\partial_1 V)
\end{split}
\end{align}
where as above the boundary integral vanishes by the boundary conditions imposed on $u$. By plugging \eqref{eq:auxmultid1} and \eqref{eq:auxmultid2} into \eqref{eq:multidstart} we obtain
\begin{align}
\label{eq:dai}
\lvg{1} \int_{\Omega} (\partial_{1}u)^2 \ge \int_{\Omega} \abs{\nabla (\partial_1u)}^2 + \int_{\Omega} V (\partial_1 u)^2 + \int_{\Omega} u (\partial_1 u)(\partial_1 V).
\end{align}
Now, if condition (i) holds, $b \cdot \nabla V = \partial_1 V = 0$ identically on $\om$ and \eqref{eq:almost1d} follows immediately from \eqref{eq:dai}. Else, if condition (ii) holds, we use integration by parts to compute
\begin{align*}
\int_{\Omega} u (\partial_1 u)(\partial_1 V) & = \frac{1}{2} \int_{\Omega} (\partial_1 u^2)(\partial_1 V) = \frac{1}{2} \int_{\Omega} \diver \begin{pmatrix} u^2 \\ \mathbf{0} \end{pmatrix} \partial_1 V  \\
& = - \frac{1}{2} \int_{\Omega} \begin{pmatrix} u^2 \\ \mathbf{0} \end{pmatrix} \cdot \nabla (\partial_1 V) + \frac{1}{2} \int_{\bnd} \partial_1 V \begin{pmatrix} u^2 \\ \mathbf{0} \end{pmatrix} \cdot \nu\,\mathrm{d}\sigma \\
& = - \frac{1}{2} \int_{\Omega} u^2 \partial_{11} V + \frac{1}{2} \int_{\bnd} \partial_1 V u^2 \nu_1 \,\mathrm{d}\sigma \\
& = -\frac{1}{2} \int_{\om} b^\top H_V b u^2 + \frac{1}{2} \int_{\ga'} (b \cdot \nabla V)(b \cdot \nu) u^2\, \mathrm{d}\sigma 
\end{align*}
where in the last step we used $u|_{\ga}=0$ and $b = \ee_1$. Condition (ii) now yields that both integrals on the right-hand side are non-negative since $V$ is concave and hence the associated Hessian matrix $H_V$ is non-positive and $(b \cdot \nabla V)(b \cdot \nu) \ge 0$ on $\ga'$. We have proved that 
\begin{equation*}
\int_{\Omega} u (\partial_1 u)(\partial_1 V)  \ge 0
\end{equation*}
which combined with \eqref{eq:dai} implies \eqref{eq:almost1d}. This concludes the proof.
\end{proof}

\section{An inequality between higher-order mixed and Dirichlet eigenvalues of Schr\"odinger operators}
\label{sec:kth ev}

In this section we compare mixed Dirichlet-Neumann eigenvalues with pure Dirichlet eigenvalues of Schr\"odinger operators. In order to do so, we make the additional assumption that the bounded, convex domain $\om \subset \rd$, $d \ge 2$, is polyhedral, and compare higher-order mixed eigenvalues with pure Dirichlet eigenvalues of the Schr\"odinger operator on such domains. In order to avoid ambiguities we give the following definition.

\begin{definition}
Let $\om \subset \rd$, $d \ge 2$, be a bounded, connected Lipschitz domain.
\begin{itemize}
\item[(i)] If $d=2$ we say that $\om$ is a polyhedral (or polygonal) domain if $\bnd$ is the union of finitely many line segments.
\item[(ii)] Recursively, if $d \ge 3$ we say that $\om$ is a polyhedral domain if for each $(d-1)$-dimensional affine hyperplane $H \subset \rd$ the intersection $H \cap \om$ is either a polyhedral domain in $\mathbb{R}^{d-1}$ (where we identify $H$ with $\mathbb{R}^{d-1}$) or empty.
\end{itemize}
\end{definition}

Let $\om \subset \rd$, $d \ge 2$, be a bounded, connected Lipschitz domain and $\ga \subset \bnd$ be relatively open and non-empty. By Rademacher's theorem the outer unit normal vector $\nu$ is well-defined almost everywhere on $\bnd$. Consequently, the $(d-1)$-dimensional tangential hyperplane
\begin{equation}
\label{eq:tg hyperplane}
T_{x} = \Big\{ \tau = (\tau_1 \ldots \tau_d)^\top \in \rd : \sum_{j=1}^d \tau_j \nu_j(x)=0 \Big\}
\end{equation}
can be defined for almost all $x \in \bnd$. Let now $\hat{\Gamma} \subseteq \ga$ denote the set of points $x \in \Gamma$ such that the tangential hyperplane $T_x$ exists; by Rademacher's theorem, this set has full measure. We then define the linear subspace
\begin{equation*}
\mathcal{S}(\Gamma) = \bigcap_{x \in \hat{\Gamma}} T_x
\end{equation*}
of $\rd$ consisting of all vectors which are tangential to almost all points of $\Gamma$. Note that $\dim \mathcal{S}(\ga) \in \{0, \ldots, d-1\}$. As observed above, it follows immediately from \cite[Theorem 4.1]{LR17} that if $V_0$ is a constant potential and $V \in L^\infty(\om)$ is real-valued then for each $k_0 \in \mathbb{N}$ there exists $\tau_0>0$ such that 
\begin{equation*}
\lambda_{k+\dim \mathcal{S}(\ga)}(V_0 + \tau V) \le \lambda_k(V_0 + \tau V)
\end{equation*}
holds for all $k \le k_0$ and $\tau \in \mathbb{R}$ such that $|\tau|<\tau_0$. Again, this type of inequality depends on the geometry of the portion $\ga$ of the boundary and on the strength of the potential and is therefore of limited relevance; in the following result we establish an inequality which instead depends on the geometry of $\ga$ and of the potential $V$ with respect to it. In order to formulate it, we denote by $\Omega'$ the set of all $x \in \Omega$ where $V$ is differentiable, which by Rademacher's theorem has full measure since $V \in W^{1,\infty}(\Omega)$ and is thus Lipschitz continuous, and define
\begin{equation*}
\big( \nabla V \big)^{\perp} = \bigcap_{x \in \Omega'} \big( \nabla V(x) \big)^{\perp} 
\end{equation*}
as the intersection of the orthogonal complements of $\nabla V(x)$ in $\rd$ for all $x \in \Omega$ where the gradient is defined classically; note that $\dim (\nabla V)^{\perp} \in \{0, \ldots, d\}$ where $\dim (\nabla V)^{\perp} = d$ if and only if $V$ is a constant potential. We then consider the set of vectors $(\nabla V)^{\perp} \cap \, \mathcal{S}(\Gamma)$ and establish an inequality
which depends on the dimension of this set, as follows.

\begin{theorem}
\label{thm:higherorder evs}
Let $\Omega \subset \rd$, $d \ge 2$, be a bounded, connected, convex, polyhedral domain, $\ga \subset \bnd$ be non-empty and relatively open and $V \in W^{1,\infty}(\Omega)$ be real-valued. Then
\begin{equation*}
\lvg{k+\dim ((\nabla V)^{\perp} \cap \, \mathcal{S}(\Gamma))} \le \lv{k}
\end{equation*}
holds for all $k \in \mathbb{N}$.
\end{theorem}

This inequality can be regarded as a unification of \cite[Theorem 4.2]{R21} and \cite[Theorem 4.1]{LR17}. Roughly speaking, it is non-trivial, i.e. $\dim ((\nabla V)^{\perp} \cap \, \mathcal{S}(\Gamma)) \ge 1$, if the potential $V$ does not depend on some of the directions which are orthogonal to the vectors tangential to $\Gamma$. 

We now provide some examples of choices of $\ga$ and $V$ where $\dim ((\nabla V)^{\perp} \cap \, \mathcal{S}(\Gamma))$ can be computed.

\begin{example}
\label{example:2d potential g}
Let $\om \subset \mathbb{R}^2$ be a convex polygon and choose $\ga$ to be one of its sides or the union of two parallel sides; we assume without loss of generality that $\om$ is rotated such that $\ga$ is parallel to the $x_2$-axis. Thus, $\mathcal{S}(\ga) = \{ (0,c)^\top : c \in \mathbb{R} \}$. Now, consider the potential $V(x) = g(x_1)$ for $x=(x_1,x_2)^\top \in \om$ where $g \in W^{1,\infty}(\{ x_1 \in \mathbb{R} : (x_1, x_2)^\top \in \om \})$ is real-valued and such that $g'$ is not identically zero. Then $\nabla V = (g'(x_1),0)^\top$ and $(\nabla V)^\perp = \mathcal{S}(\ga)$, that is, $\dim ((\nabla V)^\perp \cap \mathcal{S}(\ga)) = 1$; Theorem \ref{thm:higherorder evs} then yields that $\lambda_{k+1}^\ga(V) \le \lambda_k(V)$ holds for all $k \in \mathbb{N}$. 
\end{example}

\begin{example}
\label{example:higherdim potential g}
We can generalize Example \ref{example:2d potential g} to higher dimensions as follows. Let $\om \subset \mathbb{R}^d$ be a convex polyhedron and choose $\ga$ to be one of its faces or the union of two parallel faces; we assume without loss of generality that $\om$ can be rotated such that $\ga$ is parallel to the $\{x_2, \ldots, x_d\}$ plane. Thus, $\mathcal{S}(\ga) = \{ (0, c_1, \ldots, c_{d-1}) : c_1, \ldots, c_{d-1} \in \mathbb{R} \}$. We consider the potential $V(x) = g(x_1)$  for $x=(x_1,\ldots, x_d)^\top \in \om$ where $g \in W^{1,\infty}(\{ x_1 \in \mathbb{R} : (x_1, ,\ldots, x_d)^\top \in \om \})$ is real-valued and such that $g'$ is not identically zero. Then $\nabla V = (g'(x_1), \ldots, 0)^\top$ and $(\nabla V)^\perp = \mathcal{S}(\ga)$, from which $\dim ((\nabla V)^\perp \cap \mathcal{S}(\ga)) = d-1$. By Theorem \ref{thm:higherorder evs} we thus have that $\lambda_{k+d-1}^\ga(V) \le \lambda_k(V)$ for all $k \in \mathbb{N}$. 
\end{example}

\begin{example}
Consider a convex polyhedral domain $\om \subset  \mathbb{R}^3$ and the potential
\begin{equation*}
V(x) = a e^{b (x_1 + x_2 + x_3)}
\end{equation*}
for $x=(x_1, x_2, x_3)^\top \in \om$ and $a,b \in \mathbb{R} \setminus \{0\}$. All partial derivatives of $V$ of first order equal $b V$, thus $\nabla V \subseteq \spann \{ b (1, 1, 1)^\top V \}$ and $\dim (\nabla V)^{\perp}=3-1=2$. In particular two directions which $V$ does not depend on are the directions $(1, -1, 0)$ and $(0,1,-1)$. Thus $(\nabla V)^\perp = (1,1,1)$ and we get $\dim ((\nabla V)^\perp \cap \mathcal{S}(\ga)) = 3-1=2$ for any choice of $\ga \subset \bnd$ such that $(1,1,1) \in \mathcal{S}(\ga)$.
\end{example}

Next, we present an immediate consequence of Theorem \ref{thm:higherorder evs} where a strict inequality holds. It follows from the fact that on a bounded, connected, Lipschitz domain $\om$ the strict inequality $\lambda_k^\ga < \lambda_k^{\ga'}$ holds for any relatively open, non-empty sets $\ga \subset \ga' \subset \bnd$ such that $\ga' \setminus \ga$ has non trivial interior; the proof relies on the unique continuation principle of Lemma \ref{lem:continuation principle} and can be found for instance in \cite[Proposition 2.3]{LR17}.

\begin{corollary}
\label{cor:last section}
Let $\Omega \subset \rd$, $d \ge 2$, be a bounded, connected, convex, polyhedral domain and $V \in W^{1,\infty}(\Omega)$ be real-valued. Let $\ga \subset \bnd$ be non-empty, relatively open and let $\Sigma \subset \ga$ such that $\ga \setminus \Sigma$ has non-empty interior. Then
\begin{equation*}
\lambda^\Sigma_{k+\dim ((\nabla V)^{\perp} \cap \mathcal{S}(\Gamma))} < \lv{k}
\end{equation*}
holds for all $k \in \mathbb{N}$.
\end{corollary}

We now prove Theorem \ref{thm:higherorder evs}; we will make use of the following integration-by parts lemma. A proof can be found in \cite[Lemma A.1]{LR17}.

\begin{lemma}
Let $\om \subset \rd$, $d \ge 2$, be a polyhedral, convex domain and let $u \in H^2(\om) \cap H^1_0(\om)$. Then 
\begin{equation}
\label{eq:index id}
\int_{\Omega} (\partial_{ml} u)(\partial_{mj} u) = \int_{\Omega} (\partial_{mm} u)(\partial_{lj} u), \quad u \in \huz \cap \hd
\end{equation}
holds for all $m,l,j \in \{ 1 \ldots d\}$.
\end{lemma}

\begin{proof}[Proof of Theorem \ref{thm:higherorder evs}]
Fix $k \in \mathbb{N}$ and choose an orthogonal (in $\ld{\Omega}$) family of real-valued eigenfunctions $u_j$ of $-\Delta_{\mathrm{D}}+V$ corresponding to the Dirichlet eigenvalues of the Schr\"odinger operator $\lv{j}$, $j=1 \ldots k$. Then $u_j \in \huz \cap \hd$ by \cite[Proposition 4.8]{AGMT10} since $\Omega$ is convex. We then define
\begin{equation}
\label{eq:test functions}
\Phi = \sum_{j=1}^k a_j u_j \qquad \mathrm{and} \qquad \Psi = b^\top \nabla u_k
\end{equation}
where $a_1, \ldots, a_k \in \C$ are arbitrary and $b=(b_1, \ldots, b_d)^\top \in \mathcal{S}(\Gamma)$. Note that $\Phi \in H^1_0(\om) \cap H^2(\om)$  and $\Psi \in \hu$; note also that $\Psi$ is real-valued. The linear subspace $\mathcal{S}(\Gamma)$ is by definition tangential to almost all points of $\Gamma$ and the Dirichlet boundary condition yields
\begin{equation*}
b \cdot \nabla u_k|_{\Gamma} = 0,
\end{equation*}
that is, $\Psi \in \hmix$ and therefore $\Phi + \Psi \in \hmix$. The function $\Phi + \Psi$ is thus a suitable test function for the operator $-\Delta_\ga + V$. We divide the proof into two steps.

{\bf Step 1.} Let us denote by $\qa[\cdot]$ the quadratic form \eqref{eq:quadrform schrodinger} associated with $-\Delta +V$. We consider
\begin{equation}
\label{eq:quform}
\qa[\Phi + \Psi] = \qa[\Phi] + \qa[\Psi]  + 2 \Real \int_{\Omega} \big( \nabla\Phi \cdot \nabla \Psi + V \Phi \Psi \big).
\end{equation}
Our first aim is to evaluate the three summands on the right-hand side of \eqref{eq:quform} separately in order to get an estimate for the quadratic form $\qa[\Phi + \Psi]$. We start from the first summand,
\begin{equation}
\label{eq:firstsummand}
\qa[\Phi] = \int_{\Omega} \big( \abs{\nabla \Phi}^2 + V \abs{ \Phi}^2 \big).
\end{equation}
First,
\begin{align*}
\int_{\Omega}  \abs{\nabla \Phi}^2 & = \int_{\Omega} \nabla \Phi \cdot \overline{\nabla \Phi} = \int_{\Omega} \bigg( \sum_{l=1}^k a_l \nabla u_l \bigg) \cdot \bigg( \sum_{j=1}^k \overline{a_j} \nabla u_j \bigg) \\
& = \sum_{l,j=1}^k a_l \overline{a_j}  \int_{\Omega} \nabla u_l \cdot \nabla u_j = \sum_{l,j=1}^k a_l \overline{a_j}  \int_{\Omega} (- \Delta u_l) u_j 
\end{align*}
where we used Green's identity \eqref{eq:green} together with $u_j \in \huz$. Then,
\begin{align*}
\int_{\Omega} V \abs{ \Phi}^2  =  \int_{\Omega} V \bigg( \sum_{l=1}^k a_l u_l \bigg) \bigg( \sum_{j=1}^k \overline{a_j u_j} \bigg)   = \sum_{l,j=1}^k a_l \overline{a_j}  \int_{\Omega} V u_l u_j,
\end{align*}
and \eqref{eq:firstsummand} computes as
\begin{align*}
\qa[\Phi] & = \sum_{j=1}^k \lv{j} \abs{a_j}^2 \int_{\Omega}  \abs{u_j}^2  \le \lv{k} \int_{\Omega}  \sum_{j=1}^k  \abs{a_j u_j}^2 = \lv{k} \int_{\Omega}  \abs{\Phi}^2 
\end{align*}
due to the orthogonality of the $u_j$. As for the second summand 
\begin{equation}
\label{eq:secondsummand}
\qa[\Psi] = \int_{\Omega} \big( \abs{\nabla \Psi}^2 + V \abs{ \Psi}^2 \big),
\end{equation}
we first compute the following using identity \eqref{eq:index id} 
\begin{align*}
\int_{\Omega} \abs{\nabla \Psi}^2 & = \int_{\Omega} \nabla (b^\top \nabla u_k) \cdot \nabla (b^\top \nabla u_k)  = \sum_{m=1}^d \int_{\Omega} \sum_{l=1}^d b_l \partial_{ml} u_k  \sum_{j=1}^d b_j \partial_{mj} u_k \\
& = \sum_{m=1}^d \int_{\Omega} \partial_{mm} u_k \sum_{l,j=1}^d b_l b_j  \partial_{lj} u_k = \int_{\Omega} \Delta u_k \diver (b b^\top \nabla u_k),
\end{align*}
and then
\begin{align*}
\int_{\Omega} V \abs{ \Psi}^2 & = \int_{\Omega} V \abs{ b^\top \nabla u_k}^2 = \int_{\Omega}  V b^\top \nabla u_k b^\top \nabla u_k   \\
& = \int_{\Omega}  b^\top \nabla (V u_k) b^\top \nabla u_k   -  \int_{\Omega}  u_k b^\top \nabla V b^\top \nabla u_k  \\ 
& = \int_{\Omega}  \nabla (V u_k) \cdot b b^\top \nabla u_k   -  \int_{\Omega}  u_k b^\top \nabla V b^\top \nabla u_k  \\ 
& = - \int_{\Omega}  V u_k \diver (b b^\top \nabla u_k)   -  \int_{\Omega}  u_k b^\top \nabla V b^\top \nabla u_k 
\end{align*}
where in the last step we integrated by parts; the boundary integral vanishes since $u_j \in \huz$. The quadratic form \eqref{eq:secondsummand} thus becomes
\begin{align*}
\qa[\Psi] & = \int_{\Omega} \Delta u_k \Div (b b^\top \nabla u_k)  - \int_{\Omega}  V u_k \Div (b b^\top \nabla u_k)  -  \int_{\Omega}  u_k b^\top \nabla V b^\top \nabla u_k  \\
& = -\lv{k} \int_{\Omega} u_k \Div (b b^\top \nabla u_k)  -  \int_{\Omega}  u_k b^\top \nabla V b^\top \nabla u_k \\
& = \lv{k} \int_{\Omega} \nabla u_k \cdot b b^\top \nabla u_k  -  \int_{\Omega}  u_k b^\top \nabla V b^\top \nabla u_k  \\
& = \lv{k} \int_{\Omega} \abs{\Psi}^2 -  \int_{\Omega}  u_k b^\top \nabla V b^\top \nabla u_k 
\end{align*}
where we again integrated by parts. As for the third summand of \eqref{eq:quform}, we observe that 
\begin{align*}
- \Delta \Psi + V \Psi & = b^\top \nabla (-\Delta u_k) + b^\top \nabla (V u_k) - u_k b^\top \nabla V \\
& = b^\top \nabla (\lv{k} u_k) - u_k b^\top \nabla V \\
& = \lv{k} \Psi - u_k b^\top \nabla V
\end{align*}
holds in the distributional sense, and since by \eqref{eq:green}
\begin{equation*}
\int_{\Omega} \big( \nabla \Phi \cdot \nabla \Psi + V \Phi \Psi \big) = \int_{\Omega} \big( - \Phi \Delta \Psi + V \Phi \Psi  \big)  = \int_{\Omega} \Phi \big( - \Delta \Psi + V \Psi  \big),
\end{equation*}
we obtain
\begin{align*}
\int_{\Omega} \big( \nabla \Phi \cdot \nabla \Psi + V \Phi \Psi \big)  &  = \Phi (\lv{k} \Psi - u_k b^\top \nabla V)  = \lv{k} \int_ {\Omega} \Phi \Psi  - \int_{\Omega} \Phi u_k b^\top \nabla V.
\end{align*}
Finally, putting all three summands together we obtain the following estimate for the quadratic form $\qa[\Phi + \Psi]$:
\begin{align}
\label{eq:final estimate}
\begin{split}
& \qa[\Phi + \Psi] \le \lv{k} \int_{\Omega}  \abs{\Phi}^2  + \lv{k} \int_{\Omega} \abs{\Psi}^2  -  \int_{\Omega}  u_k b^\top \nabla V b^\top \nabla u_k  \\ 
& \phantom{\qa[\Phi + \Psi] \le} \, - 2  \lv{k} \Real \Big( \int_ {\Omega} \Phi \Psi   - \int_{\Omega} \Phi u_k b^\top \nabla V \Big)  \\
& \phantom{\qa[\Phi + \Psi]} \, = \lv{k} \int_{\Omega}  \abs{\Phi + \Psi}^2   - 2 \Real \int_{\Omega} \Phi u_k b^\top \nabla V \\
& \phantom{\qa[\Phi + \Psi] \le} \,  - \int_{\Omega} u_k b^\top \nabla V b^\top \nabla u_k. 
\end{split}
\end{align}
If we now choose $b \in (\nabla V)^{\perp}$, then the last two summands on the right-hand side of \eqref{eq:final estimate} vanish, and we are left with the inequality
\begin{equation}
\label{eq:super final}
\qa[\Phi + \Psi] \le \lv{k} \int_{\Omega}  \abs{\Phi + \Psi}^2. 
\end{equation}

{\bf Step 2.} In order to apply the min-max principle \eqref{eq:minmax}, our aim is to estimate the dimension of the linear space consisting of functions of the form $\Phi + \Psi$ as in \eqref{eq:test functions} where $b \in \mathcal{S}(\Gamma) \cap (\nabla V)^\perp$. In order to do so, we first estimate the dimension of the linear space of functions $\Phi + \Psi$ without any restriction on the choice of $b$. We start by claiming that
\begin{equation}
\label{eq:dim claim}
\dim \big( \lspan \{ u_1, \ldots, u_k, \partial_1 u_k, \ldots, \partial_d u_k \} \big) = k + \dim \big( \lspan \{ \partial_1 u_k, \ldots, \partial_d u_k \} \big).
\end{equation}
By assumption, $\dim \big( \lspan \{ u_1, \ldots, u_k \} \big)=k$. Let $\omega \in \lspan \{ u_1, \ldots, u_k \} \cap \linebreak  \{ \partial_1 u_k, \ldots, \partial_d u_k \}$: then $\omega \in \huz$ and $\omega = \sum_{j=1}^d b_j \partial_j u_k$ for some $b_1 \ldots b_d \in \mathbb{R}$. For a contradiction, assume that the vector $(b_1, \ldots, b_d)^{\top}$ is non-trivial. Let $\Lambda$ be a face of $\bnd$ such that $(b_1, \ldots, b_d)^{\top}$ is not tangential to $\Lambda$, and let $\tau^1, \ldots, \tau^{d-1}$ be $d-1$ linearly independent tangential vectors to $\Lambda$. Then, the system $\{ \tau^1, \ldots, \tau^{d-1}, \linebreak (b_1, \ldots, b_d)^{\top} \}$ is linearly independent. The Dirichlet boundary condition $u_k|_{\Lambda}=0$ implies that
\begin{equation*}
\tau^j \cdot \nabla u_k|_{\Lambda} = 0, \qquad j=1, \ldots, d-1,
\end{equation*} 
while $\omega \in \huz$ yields that 
\begin{equation*}
\omega|_{\Lambda} = (b_1, \ldots, b_d)^{\top} \cdot \nabla u_k|_{\Lambda} = 0.
\end{equation*}
Since the constant outer unit normal $\nu$ on $\Lambda$ can be written as a linear combination of $\{ \tau^1, \ldots, \tau^{d-1} \}$ and $(b_1, \ldots, b_d)^{\top}$, these two identities imply that 
\begin{equation*}
\nu \cdot \nabla u_k|_{\Lambda} = \partial_\nu u_k|_{\Lambda} = 0,
\end{equation*}
which together with $u_k|_{\Lambda} = 0$ implies that $u_k=0$ by Lemma \ref{lem:continuation principle}, a contradiction. Thus $(b_1, \ldots, b_d)^{\top}=0$ and $\omega = 0$, from which we conclude \eqref{eq:dim claim}. Next, note that the partial derivatives $\partial_1 u_k, \ldots, \partial_d u_k$ are linearly independent. To this purpose let $b_1, \ldots, b_d \in \mathbb{C}$ be such that 
\begin{equation}
\label{eq: aux1}
\sum_{j=1}^d b_j \partial_j u_k=0
\end{equation}
on $\Omega$ and assume for a contradiction that we are off the case $b_1 = \ldots = b_d = 0$. We can assume without loss of generality that the vector $(b_1, \ldots, b_d)^{\top}$ is non-trivial so that \eqref{eq: aux1} implies that the derivative of $u_k$ in its direction vanishes on all of $\Omega$. This combined with $u_k|_{\bnd}=0$ yields that $u_k=0$ on $\Omega$, a contradiction. In particular, linearly independent vectors $(b_1, \ldots, b_d)^\top \in \mathcal{S}(\Gamma)$ lead to linearly independent functions $\sum_{j=1}^d b_j \partial_j u_k \in \hmix$. Therefore,
\begin{equation}
\label{eq:tmp dim}
\dim \big(\Psi ~\mathrm{of}~\mathrm{the}~\mathrm{form}~\eqref{eq:test functions} : b \in \mathcal{S}(\Gamma) \big) \ge \dim\,\mathcal{S}(\Gamma).
\end{equation}
Let us now go back to estimate \eqref{eq:super final}. Since we are restricting ourselves to vectors $b=(b_1, \ldots, b_d)^\top$ which belong both to $\mathcal{S}(\Gamma)$ (so that $\Psi \in \hmix$) and to $(\nabla V)^{\perp}$ (so that the two summands in \eqref{eq:final estimate} can be nullified), estimate \eqref{eq:tmp dim} refines to
\begin{equation*}
\dim \big(\Psi ~\mathrm{of}~\mathrm{the}~\mathrm{form}~\eqref{eq:test functions} : b \in (\nabla V)^{\perp} \cap \mathcal{S}(\Gamma) \big) \ge \dim \big((\nabla V)^{\perp} \cap \mathcal{S}(\Gamma) \big)
\end{equation*}
which combined with \eqref{eq:dim claim} finally yields the desired estimate
\begin{equation*}
\dim \big( \Phi + \Psi ~\mathrm{of}~\mathrm{the}~\mathrm{form}~\eqref{eq:test functions} : b \in (\nabla V)^{\perp} \cap \mathcal{S}(\Gamma) \big) \ge k+\dim \big( (\nabla V)^{\perp} \cap \, \mathcal{S}(\Gamma) \big).
\end{equation*}
This means that
\begin{equation*}
\qa[u] \le \lv{k} \int_{\Omega}  \abs{u}^2
\end{equation*}
for all $u$ in a subspace of $\hmix$ of dimension $k+\dim \big( (\nabla V)^{\perp} \cap \, \mathcal{S}(\Gamma) \big)$ or larger. This combined with the min-max principle \eqref{eq:minmax} yields the assertion of the theorem.
\end{proof}

\section*{Acknowledgements}

I would like to sincerely thank my supervisor Jonathan Rohleder for his many comments and suggestions all along this work, and Hans Oude Groeniger for the several fruitful discussions. I am also grateful to Monique Dauge for the insightful exchange on $H^2$-regularity. This research was funded by grant no. 2018-04560 of the Swedish Research Council (VR).

\end{document}